\documentclass[a4paper,oneside,10pt]{article}

\date{}

\usepackage{lipsum}
\usepackage{epstopdf}
\usepackage{algorithmic}
\usepackage[latin1]{inputenc}
\usepackage[T1]{fontenc}
\usepackage[margin=2cm]{geometry}
\usepackage{amsmath,amssymb,amsthm,amsfonts}
\usepackage{indentfirst}
\usepackage{graphicx}
\usepackage{enumitem}
\usepackage{verbatim}
\usepackage{mathrsfs}
\usepackage{tikz}
\usepackage[english]{babel}
\usepackage{url}
\usepackage{bbm}
\usepackage[toc,page]{appendix}
\usepackage[colorlinks,allcolors=blue]{hyperref}

\usetikzlibrary{matrix}
\usetikzlibrary{arrows}
\usetikzlibrary{graphs}
\theoremstyle{plain}
\newtheorem{thm}{Theorem}[section]
\newtheorem*{thm*}{Theorem}
\newtheorem{teor*}{Theorem}
\newtheorem{cor}[thm]{Corollary}
\newtheorem{lem}[thm]{Lemma}
\newtheorem{prop}[thm]{Proposition}
\theoremstyle{definition}
\newtheorem{defn}{Definition}[section]
\theoremstyle{remark}
\newtheorem{rem}{Remark}
\numberwithin{equation}{section}
\DeclareMathOperator{\sgn}{sgn}

\title{Attainable profiles for conservation laws\\ 
with flux function spatially discontinuous at a single point}

\author{Fabio Ancona$^1$
	\and Maria Teresa Chiri$^2$
	}

\date{%
	\small\textit{$^1$Dipartimento di Matematica "Tullio Levi-Civita", Universit\`a di Padova, Italy \\(ancona@math.unipd.it)\\%
	$^2$Department of Mathematics, Penn State University, University Park, PA, USA \\	(mxc6028@psu.edu)}\\[2ex]%
	\today
}

\begin{document}
\maketitle
\begin{abstract}
Consider a scalar  conservation law with discontinuous flux
\begin{equation*}\tag{1}
\quad u_{t}+f(x,u)_{x}=0, 
\qquad f(x,u)= \begin{cases}
f_l(u)\ &\text{if}\ x<0,\\
f_r(u)\ & \text{if} \ x>0,
\end{cases}
\end{equation*}
where $u=u(x,t)$ is the state variable and $f_{l}$, $f_{r}$ are strictly convex maps. We study the Cauchy problem for (1) 
from the point of view of control theory regarding the initial datum as a control. Letting $u(x,t)\doteq \mathcal{S}_t^{AB} \overline u(x)$ denote the solution of the 
Cauchy problem for (1), with initial datum $u(\cdot,0)=\overline u$, that satisfy at $x=0$ the interface entropy condition associated to a connection $(A,B)$ (see~\cite{MR2195983}), we analyze the family of 
profiles that can be attained by (1) at a given time $T>0$:
\begin{equation*}
\mathcal{A}^{AB}(T)=\left\{\mathcal{S}_T^{AB} \,\overline u : \ \overline u\in{\bf L}^\infty(\mathbb{R})\right\}.
\end{equation*}
We provide a full characterization of $\mathcal{A}^{AB}(T)$ as a class of functions in $BV_{loc}(\mathbb{R}\setminus\{0\})$ that satisfy suitable Ole\v{\i}nik-type inequalities,
and that admit one-sided limits at $x=0$ which satisfy specific conditions related to the interface entropy criterium.
Relying on this characterisation, we establish the ${\bf L^1}_{loc}$-compactness of the set of attainable
profiles when the initial data $\overline u$ vary in a given class of uniformly bounded functions, taking values in closed convex sets.
We also discuss some applications of these results to optimization problems arising in 
porous media flow models for oil recovery and in traffic flow.
\end{abstract}
\section*{Introduction}

Consider the  Cauchy problem for the scalar conservation law in one space dimension
 \begin{align}
 u_{t}+f(x,u)_{x} & =0,\qquad x\in \mathbb{R},\quad t\geq 0, \label{eq1}\\
 u|_{t=0} & =\overline u,\qquad x\in \mathbb{R}, \label{datum}
 \end{align}
where $u=u(x,t)$ is the state variable, and the flux $f(x,u)$ is a discontinuous function given by
 \begin{equation}
 \label{flux}
 f(x,u)=
 \begin{cases}
 f_{l}(u)&\quad\text{if}\qquad x<0\,,
 \\
 \noalign{\smallskip}
 f_{r}(u)&\quad\text{if}\qquad x>0\,, 
 \end{cases}
 \end{equation}
with $f_{l}$ and $f_{r}$ smooth, strictly convex maps.
The equation~\eqref{eq1} is usually supplemented with appropriate coupling conditions
imposed at the point of discontinuity of the flux 
so to guarantee uniqueness
of solutions to the Cauchy problem~\eqref{eq1}-\eqref{datum}.
Namely, 
the traces 
\begin{equation}
\label{traces}
u_{l}(t)=\lim_{t\rightarrow 0-}{u(x,t)},\qquad u_{r}(t)= \lim_{t\rightarrow 0+}{u(x,t)},
\end{equation}
of a weak distributional solution of~\eqref{eq1}, \eqref{flux}, must satisfy the Rankine-Hugoniot condition 
\begin{equation}
\label{RH-i}
 f_{l}(u_l(t))=f_{r}(u_r(t))
 \qquad\text{for a.e.} \quad t>0\,,
 \end{equation}
at the interface $x=0$.
Moreover, 
various type of  admissibility conditions (interface entropy conditions)
imposed on $u_{l,r}$ have been introduced in the literature, 
according with different modelling assumptions (see~\cite{MR3416038,MR3369104}).
Such conditions lead to different solutions of the Cauchy problem~\eqref{eq1}-\eqref{datum},
which are appropriate for the particular physical phenomena modelled by~\eqref{eq1}.
Alternatively, one can equivalently characterize the admissible solutions in terms 
of  Kru\v{z}kov-type (possibly singular) entropy inequalities satisfied up-to-the flux-discontinuity
interface (cfr.~\cite{MR2024741}),
or using  extended families of
entropy inequalities associated to the so called {\it partially adapted entropies} (see~\cite{MR3416038,andreianov:hal-02197482,MR2132749,MR2505870}).

Starting with the works by Isacson \& Temple~\cite{MR861520}  and by Risebro and collaborators~\cite{MR1109304,MR1158825,MR1710013},
conservation laws with discontinuous flux has been an intense subject of research
in the last three decades
(e.g. see~\cite{MR2807133,MR3870565}
and references therein). 
Solutions of~\eqref{eq1}, \eqref{flux}, satisfying the above mentioned admissibility criteria, can be obtained 
as limit of approximations
constructed 
by regularization of the flux~\cite{MR2209759,MR1710013,MR1912069}, by wave front-tracking~\cite{MR2291816,MR1158825},
by Godunov method~\cite{MR2195983,MR3732666} 
and several other numerical schemes~\cite{MR2505870,MR2086124,MR1961002}
or by vanishing viscosity~\cite{MR2670658,MR3870565}.
In particular, in~\cite{MR2807133} it was set up a general framework
that encompasses all the notions of admissible solutions to the Cauchy problem~\eqref{eq1}-\eqref{datum}
which lead to the existence of an ${\bf L^1}$-contractive semigroup.

In this paper we study the system~\eqref{eq1}-\eqref{datum} from the point of view of control theory,
regarding the initial data $u_0$ as a control.
Namely, 
we provide a characterization of the space-profile configurations that can be attained at any fixed time $T>0$\,:
\begin{equation*}
\mathcal{A}(T
)=\big\lbrace u(\cdot,T): u\text{ is an admissible solution of (\ref{eq1}),\eqref{flux}-(\ref{datum}) with } u_{0}\in
\mathbf{L}^{\infty}(\mathbb{R})
\big\rbrace.
\end{equation*} 
Here, $u$ is a solution of~\eqref{eq1}-\eqref{datum} satisfying an interface entropy condition 
associated to a so-called {\it interface-connection} $(A, B)$~\cite{MR2195983,MR2505870}.
A connection $(A, B)$ is a pair of states connected by a stationary weak solution
of (\ref{eq1}),\eqref{flux},
taking values $A$ for $x<0$, and $B$ for $x>0$, which has characteristics diverging
from (or parallel to) the flux-discontinuity interface $x=0$. Such a solution characterize the possible {\it undercompressive}
(or {\it marginally undercompressive}) shock waves exhibited by admissible solutions of (\ref{eq1}),\eqref{flux}
that satisfy an interface entropy condition involving the connection $(A, B)$
(cfr.\cite{MR2195983,MR2505870}). 
The reason for choosing this type of admissibility conditions for solutions of (\ref{eq1}),\eqref{flux} 
is twofold. On one hand,
it is consistent with the models of two-phase flows in heterogeneous porous medium~\cite{MR2195983} 
or of traffic flow on roads with variable surface conditions~\cite{MR898784}.
On the other hand, it allows to treat any connection $(A,B)$ as a pair of control parameters
as well. 

We show that any element in $\mathcal{A}(T)$  belongs to a class of functions in $BV_{loc}(\mathbb{R}\setminus\{0\})$
(with locally bounded variation on $\mathbb{R}\setminus\{0\}$),
which: 
\vspace{-5pt}
\begin{itemize}
\item[-] satisfy suitable Ole\v{\i}nik-type inequalities involving the first and second derivatives
of the maps~$f_l, f_r$;\\
\vspace{-18pt}
\item[-] admit one-sided limits at $x=0$ which satisfy specific conditions related to the interface entropy criterium
of the $(A, B)$-connection.
\end{itemize}
\vspace{-7pt}
Viceversa, we establish  an exact-time controllability result, i.e. we prove that, 
for any target function $\omega$ of the aforementioned class, there exist an initial datum $\overline u$ 
and a connection $(A, B)$
that steer the system~\eqref{eq1}-\eqref{datum} to~$\omega$ at a given time~$T$.
These results extend to the spatially-discontinuous setting the characterization of the attainable 
profiles established in~\cite{MR1616586,MR1663615,MR1612027} for conservation laws with convex flux depending only on the state variable.
Such results 
are obtained 
exploiting, as in~\cite{MR1616586}, the theory of generalized characteristics,
which was developed by Dafermos~\cite{MR0457947}
for conservation laws with convex flux (in the state variable) depending smoothly on the space variable.
A detailed analysis of the structure of admissible solutions for a 
given connection (cfr.~Proposition~\ref{Norarefactions} and Remark~\ref{charact-lrtraces-admissible}) is also fundamental to derive a full characterization of the attainable profiles.

Hyperbolic partial differential equations with discontinuous coefficients 
arise in 
many different applications in physics and engineering including: two-phase flow models in porous media
with changing rock types (for oil reservoir simulation)~\cite{MR1158825,GR93}; slow erosion granular flow models~\cite{MR3384834};
clarifier-thickener problems of continuous sedimentation (in waste-water treatment plants)~\cite{MR2136036,MR1381652};
population-balance models of steel ball wear in grinding mills~\cite{MR2133239};
ion etching in semiconductor industry~\cite{MR948078};
 traffic flow models
 with roads of varying amplitudes or surface conditions
 ~\cite{MR898784};
Saint Venant models of blood flow in endovascular treatments~\cite{MR1901663,Ca};
radar shape-from-shading models~\cite{MR1912069}.
This kind of equations appear also in the analysis of inverse problems~\cite{MR2396486,MR1680830} 
or of optimal control problems~\cite{MR3238137} 
for conservation laws with smooth flux,
where one 
needs to deal with the 
backward adjoint transport equation 
with discontinuous coefficients, which depend on the (possibly discontinuous) solution
of the conservation law. 
Moreover, conservation laws with discontinuous flux 
arise also as a reformulation of
balance laws~\cite{MR2396489} or of triangular systems of conservation laws~\cite{MR3870565,MR2396488},
in order to design efficient numerical schemes or to analyse their well-posedness.
Finally, we observe that such a class of PDEs share fundamental features 
of conservation laws evolving on simple networks composed by a number of edges
connected together by a junction~\cite{MR3553143,MR2291816}, 
which is a topic 
attracting a vast interest in the last twenty-five years
for the wide range of applications~\cite{MR3200227}.

Despite a large amount of literature on the theoretical and numerical aspects of conservation
laws with discontinuous flux produced in the last three decades, almost no investigation of 
control issues for such a class of PDEs has been performed so far.
The goal of the present paper is to provide a first step toward the analysis of controllability
properties of these type of equations. 
Having in mind applications to optimization problems,
we rely on the characterization of the attainable profiles to establish  compactness in the $\mathbf{L}^1$-topology
of the attainable set in connection with 
classes of uniformly bounded initial data taking values in closed convex sets. 
We then apply these results 
to two classes of optimization problems for porous media flow in oil recovery and for traffic flow, where one 
is interested in: 
\vspace{-5pt}
\begin{itemize}
\item[-] minimising the distance from a target configuration (for both models)
 or  the fuel consumption
in a given road segment (for the latter model);\\
\vspace{-18pt}
\item[-] maximising the net present value of the waterflooding process (in the first model). 
\end{itemize}
We point out that a further	 step in the research direction pursued in this paper is the characterization of the traces of admissible solutions
at the the flux-discontinuity interface as well as the analysis of the reachable set when one 
fixes the initial data and considers such traces as control parameters (cfr.~\cite{MR3857883} within the network setting), which is the object of the forthcoming paper~\cite{AC1}.

The paper is organized in the following way. In Section 1 we recall the definition of interface 
entropy condition relative to an interface connection $(A, B)$, and the corresponding definition of
$AB$-entropy solution. We also review the well-posedness theory of $\mathbf{L}^1$-contractive
semigroups for this particular class of entropy admissible solutions.
Section 2 collects the statements of the main results on the full description of the set of attainable profiles and their topological properties. In Section 3 we establish a preliminary lemma
concerning the
structure of $AB$-entropy solutions.
The proofs of the characterization of the attainable set and of its compactness is provided
in Section 4 and Section 5, respectively.
Finally, in Section 6 we discuss two applications arising in 
traffic flow models, which lead to
variational problems
with cost functionals depending on the profile of the solutions,
where we regard as control parameters both the initial data
and the connection states. 

\section{Preliminaries and setting of the problem}
\label{sec-prelim}

Consider the scalar conservation law $(\ref{eq1})$ with flux given by $(\ref{flux})$,
and assume that $f_l, f_r$ 
coincide at two points
of their domain which, up to a reparametrization of the unknown variable, we may suppose to be $u=0$ and $u=1$. 
Observe that, by strict convexity, $f_{l}$ $f_{r}$ admit a unique point of minimum 
which we call, respectively, $\theta_{l}$ and $\theta_{r}$.
Hence, we shall make the following standing hypotheses on the flux $f$ in~\eqref{flux}:
\begin{enumerate}
\item[{\bf H1)}] $f_l, f_r: \mathbb{R}
\rightarrow \mathbb{R}$ 
are 
twice continuously differentiable, (uniformly) strictly convex maps
$$\min\big\{f''_{l}(u)\,\, f''_r(u)\big\}\geq c>0\qquad \forall~u\in \mathbb{R};$$
\item[{\bf H2)}] $f_l(0)=f_r(0)$, \, $f_l(1)=f_r(1)$;
\item[{\bf H3)}] 
$\theta_{l}\geq 0, \ \theta_r\leq 1$.
\end{enumerate}

We recall that, regardless of how smooth the initial data are, 
nonlinear conservation laws as~\eqref{eq1}, \eqref{flux} do not
posses in general classical solutions globally defined 
in time, even when $f_l=f_r$, since they can develop discontinuities (shocks) 
in finite time. Hence, it is natural to consider 
weak solutions in the sense of distributions  that, for sake of uniqueness, satisfy 
the classical Kru\v{z}kov entropy inequalities
away from the point of the flux discontinuity, and a further interface entropy condition
at the flux-discontinuity interface. 
As observed in the introduction, for modellization and control treatment reasons,
we shall employ an admissibility condition involving the so-called interface connection
introduced in~\cite{MR2195983},
which can be equivalently formulated in terms of an interface entropy condition
or of extended entropy inequalities adapted to the particular connection taken into account
(cfr.~\cite{MR2195983,MR2807133,MR2505870}).

\begin{defn}
\label{int-conn-def}
{\bf (Interface Connection)} Let $(A,B)\in\mathbb{R}^{2}$. Then $(A,B)$ is called a connection (Fig. 1) if it satisfies:
\begin{itemize}
\item[(i)] $f_{l}(A)=f_{r}(B)$;\qquad\qquad (ii) $A\leq \theta_l, \ B\geq \theta_r$.
\end{itemize}
We shall denote with $\mathscr{C}_f$ the set of pairs of connections
associated to the flux $f(x,u)$ in~\eqref{flux}.
\end{defn}
\begin{figure}
\begin{center}

\begin{tikzpicture}[scale=1.5,cap=round]
  \def\costhirty{0.8660256}

  \tikzstyle{axes}=[]
  \tikzstyle{important line}=[very thick]
  \tikzstyle{information text}=[rounded corners,fill=red!10,inner sep=1ex]

    \draw[->] (-2,0) -- (2.5,0) node[right] {$u$};
    \draw (-1,5/2) node[anchor=east] {$f_{l}$};
    \draw [black,domain=-1:1] plot (\x, {2*pow(\x,2)+0.5});
    \draw [black,domain=0:1.8] plot (\x, {2*pow(\x,2)-3.6*\x +2.5});
    \draw (1.8,5/2) node[anchor=west] {$f_{r}$};
    \draw [black,domain=-1:1.8] plot (\x, {3/2});
    \draw[thick,dashed] (-0.707,3/2) -- (-0.707,0);
    \draw[thick,dashed] (0.707,3/2) -- (0.707,0);
    \draw[thick,dashed] (0.343,3/2) -- (0.343,0);
    \draw[thick,dashed] (1.456,3/2) -- (1.456,0);
    \draw[thick,dashed] (0,1/2) -- (0,0);
    \draw[thick,dashed] (0.9,0.88) -- (0.9,0);
    \draw[blue,ultra thick] (-0.707,3/2) -- (1.456,3/2);
  \draw (-0.707,0) node[anchor=north] {$A$};
  \draw (0.707,0) node[anchor=north] {$\bar{A}$};
  \draw (1.456,0) node[anchor=north] {$B$};
  \draw (0.343,0) node[anchor=north] {$\bar{B}$};
  \draw (0,0) node[anchor=north] {$\theta_{l}$};
  \draw (0.9,0) node[anchor=north] {$\theta_{r}$};
  \end{tikzpicture}
  \caption{Example of  AB connection with $f_{l}$, $f_{r}$ strictly convex fluxes }
  \end{center}
\end{figure}
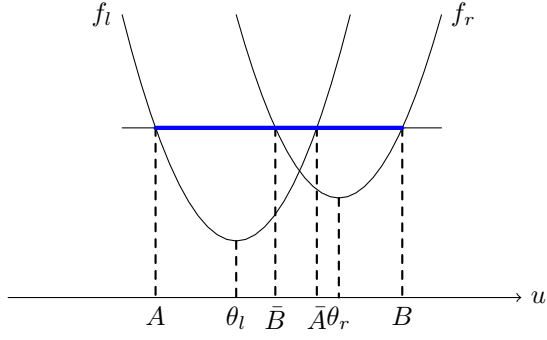
%
Observe that condition (ii) is equivalent to:\ \  (ii)' $f'_{l}(A)\leq 0$ and $f'_{r}(B)\geq 0$;\
which shows that the function
\begin{equation}
\label{AB-adapt-entr}
k_{{AB}}(x)=
\begin{cases}
A\ &\text{if} \ \ x<0\,,
\\
B\ &\text{if} \ \ x>0\,,
\end{cases}
\end{equation}
is a stationary undercompressive
(or marginally undercompressive) weak solution of~\eqref{eq1},\eqref{flux}, since
its characteristics diverge
from (or are parallel to) the flux-discontinuity interface $x=0$. 
The function $k_{AB}$ is used in~\cite{MR2505870} to define the adapted entropy
$\eta_{AB}(x,u)=\big|u-k_{AB}(x)\big|$, which in the spirit of~\cite{MR2132749} is employed  to select a unique
solution of the Cauchy problem~\eqref{eq1},\eqref{flux}-\eqref{datum}, according with the following definition.
\begin{defn}
\label{def-ab-entr-sol}
{\bf ($\textbf{\textit{AB}}$-Entropy Solution)}
Let $(A,B)$ be a connection and let $k_{AB}$ be the function defined in~\eqref{AB-adapt-entr}.
A function $u\in \mathbf{L}^{\infty}(\mathbb{R}\times [0,+\infty))$ is said  an $AB$-entropy solution 
of~\eqref{eq1},\eqref{flux}-\eqref{datum} if the following holds:
\begin{itemize}
\item[(i)] $u$ is a weak distributional solution of~\eqref{eq1},\eqref{flux} on $\mathbb{R}\times \mathbb{R}_+$, that is,
for any test function $\phi\in\mathcal{C}^1_c$ with compact support 
contained in $\mathbb{R}\times (0,+\infty)$, there holds
\begin{equation*}
 \int_{-\infty}^{\infty}\int_{0}^{\infty}
 \big\{u \phi_t+f(x,u)\phi_x\big\}dxdt
 =0.
 \end{equation*}
\item[(ii)]$u$ is a Kru\v{z}hkov entropy weak solution of~\eqref{eq1},\eqref{flux}-\eqref{datum}
on $(\mathbb{R}\setminus\{0\})\times[0,+\infty)$, that is $t\to u(\cdot,t)$ is a continuous map from $[0,+\infty)$
in $\mathbf{L}^1_{loc}(\mathbb{R})$, the initial condition~\eqref{datum} is satisfied, and:
\begin{itemize}
\item[(ii.a)]
 for any non-negative test function $\phi\in\mathcal{C}^1_c$ with compact support 
contained in $(-\infty,0)\times (0,+\infty)$, there holds
\begin{equation*}
 \int_{-\infty}^0\int_{0}^{\infty}
 \big\{|u-k|\phi_t+\big(f_l(u)-f_l(k)\big)\sgn(u-k)\phi_x
 \big\}dxdt
\geq 0\qquad\forall~k\in\mathbb{R};
 \end{equation*}
\item[(ii.b)]
 for any non-negative test function $\phi\in\mathcal{C}^1_c$ with compact support 
contained in $(0,+\infty)\times (0,+\infty)$, there holds
\begin{equation*}
 \int_0^{+\infty}\int_{0}^{\infty}
 \big\{|u-k|\phi_t+\big(f_r(u)-f_r(k)\big)\sgn(u-k)\phi_x
 \big\}dxdt
\geq 0\qquad\forall~k\in\mathbb{R}.
 \end{equation*}
 \end{itemize}
\item[(iii)] $u$ satisfies a Kru\v{z}hkov-type entropy inequality relative to the connection $(A,B)$,
that is,
for any non-negative test function $\phi\in\mathcal{C}^1_c$ with compact support 
contained in $\mathbb{R}\times (0,+\infty)$, there holds
\begin{equation*}
 \int_{-\infty}^{+\infty}\int_{0}^{\infty}
 \big\{\big|u-k_{AB}(x)\big|\phi_t+\big(f(x,u)-f(x,k_{AB}(x))\big)\sgn(u-k_{AB}(x))\phi_x
 \big\}dxdt
\geq 0.
 \end{equation*}
\end{itemize}
\end{defn}
\begin{rem}
\label{propert1-AB-sol}
If $u$ is an $AB$-entropy solution,
by property (ii) and because of the strict convexity of the fluxes~$f_{l,r}$,
it follows that  $u(\cdot,t)\in  BV_{loc}(\mathbb{R}\setminus\{0\})$ for any $t>0$. 
Actually, it was shown in~\cite{MR2743877} that for all connections such that both 
$A\neq \theta_{l}$ and $B\neq \theta_{r}$, 
one has  $u(\cdot,t)\in  BV_{loc}(\mathbb{R})$ for any $t>0$.
On the other hand, when $(A,B)$ is a critical connection, i.e. when either $A=\theta_{l}$ or $B= \theta_{r}$,
 the total variation of $u(\cdot,t)$ may well blow up
in a neighbourhood of the flux-discontinuity interface $x=0$, 
at some time $t>0$ (see~\cite{MR2743877}).
However,  since  $u$ is  in particular
a distributional solution of $u_t +f_l(u)_x=0$ on $(-\infty,0)\times(0,+\infty)$,
and of $u_t +f_r(u)_x=0$ on $(0,+\infty)\times (0,+\infty)$,
and since the fluxes $f_{l,r}$ are strictly convex,
relying on a result in~\cite{MR2374223} (see also~\cite{MR1869441})
one deduces
that $u(\cdot,t)$ still admits strong left and right traces at $x=0$, i.e. that 
(after a possibly modification on a set of measure zero)
for all $t>0$
there exist the one-sided
limits~\eqref{traces} (cfr~\cite{MR2505870}). Hence, since $u$ is a distributional solution of~\eqref{eq1},\eqref{flux}
on $\mathbb{R}\times (0,+\infty)$, by property (i), it follows that the Rankine-Hugoniot condition~\eqref{RH-i}
holds.
Furthermore, by the analysis in ~\cite[Lemma 3.2]{MR2505870} and ~\cite[Section~4.8]{MR2807133},
it follows that, because of
condition (i) of Definition~\ref{int-conn-def} 
and  assumption {\bf H1)} on $f_l, f_r$, we can equivalently
replace condition~(iii) in Definition~\ref{def-ab-entr-sol} with
\begin{itemize}
\item[(iii)'] $u$ satisfies an interface entropy condition relative to the connection $(A,B)$,
that is, the one-sided limits~\eqref{traces} satisfy
\begin{equation}
\label{int-entr-cond3}
\begin{gathered}
f_{l}(u_l(t))=f_{r}(u_r(t))\geq f_{l}(A)=f_{r}(B) ,\\
\noalign{\smallskip}
\big(u_l(t)\leq \theta_l\ \ \ \text{and} \ \ \ u_r(t) \geq \theta_r\big)\quad
 \Longrightarrow \ \ (u_l(t),u_r(t))= (A,B)
 \end{gathered}
\ \qquad \text{for a.e.}\ t>0\,.
\end{equation}
\end{itemize}
The  first condition in~\eqref{int-entr-cond3} prescribes
 that the flux of the solution at the flux-discontinuity interface be greater or equal than the value of the flux on the $(A,B)$ connection. Whereas the second condition in~\eqref{int-entr-cond3} 
excludes that the characteristics diverge from the flux-discontinuity interface 
when $(u_l(t),u_r(t))\neq (A,B)$,
i.e. the {\it $(A,B)$ characteristic condition} in~\cite[Definition~1.4]{MR2505870}
is verified.
\end{rem}
\begin{rem}
\label{lax-kruzhkov-entr}
Since $f_l, f_r$ are strictly convex maps, the Kru\v{z}hkov entropy
inequalities (ii.a)-(ii.b) in Definition~\ref{def-ab-entr-sol} are equivalent to the Lax entropy condition
~\cite{MR0267257,MR93653} 
\begin{equation}
u(x-,t)\geq u(x+,t)\qquad\forall~t,x>0\,.
\end{equation}
\end{rem}

It was proved in~\cite{MR2195983,MR2505870} (see also~\cite{MR2807133}) that  $AB$-entropy solutions of~\eqref{eq1},\eqref{flux}
with bounded initial data are unique and form an $\mathbf{L}^1$-contractive semigroup.  
We collect the properties of such a semigroup
in the following 
\begin{thm}
\label{semigroup-AB}
{\bf (Semigroup of $\textbf{\textit{AB}}$-Entropy Solutions)}~{\rm \cite{MR2195983,MR2505870}}
Let $f$ be a flux as in~\eqref{flux} satisfying the assumptions {\bf H1), H2), H3)}.
Then, given a connection $(A,B)\in\mathscr{C}_f$,
there exists a map $$\mathcal{S}^{AB} : [0,+\infty)\times 
 \mathbf{L}^\infty(\mathbb{R}) \to \mathbf{L}^\infty(\mathbb{R}),
 \qquad
 (t, \overline u) \mapsto \mathcal{S}_t^{AB} \overline u\,,$$
enjoying the following properties:  
\begin{itemize}
\item[(i)]
For each $\overline u\in\mathbf{L}^\infty(\mathbb{R})$, the function $u(x,t)\doteq \mathcal{S}_t^{AB} \overline u(x)$
provides the unique $AB$-entropy solution of the Cauchy problem~\eqref{eq1},\eqref{flux}-\eqref{datum}.
\item[(ii)]
$$\mathcal{S}_0^{AB} \overline u = \overline u,\qquad \mathcal{S}_s^{AB} \circ \mathcal{S}_t^{AB} \overline u = \mathcal{S}_{s+t}^{AB} \,\overline u
\qquad\forall~t,s\geq 0,\quad \forall~\overline u\in\mathbf{L}^\infty(\mathbb{R})\,.$$
\item[(iii)]
$$\big\| \mathcal{S}^{AB}_t \overline u -  \mathcal{S}^{AB}_s \overline v\big\|_{\mathbf{L}^1}\leq \big\|\overline u-\overline v\big\|_{\mathbf{L}^1}
+ L\big|t-s\big|\qquad\forall~t,s\geq 0,\quad \forall~\overline u, \overline v\in\mathbf{L}^\infty(\mathbb{R})\,,$$
for some positive constant $L>0$.
\end{itemize}
\end{thm}

In the present paper we regard as control parameters both the initial data and the connection states
whose flux provides a lower bound on the flux of the solution at the flux-discontinuity interface.
Then, given a set $\mathcal{U}\subset\mathbf{L}^\infty(\mathbb{R})$, and a set $\mathscr{C}\subset \mathscr{C}_f$  of connections, 
we consider the following {\it attainable sets} for~\eqref{eq1},\eqref{flux}:
\begin{equation}
\label{att-set-def1}
\mathcal{A}^{AB}\big(T,\mathcal{U}\big)\doteq  
\big\{\mathcal{S}^{AB}_t \overline u : \overline u\in\mathcal{U}\big\}\,,\qquad
\mathcal{A}\big(T,\mathcal{U},\mathscr{C}\big)\doteq
\bigcup_{(A,B)\in\mathscr{C}} \mathcal{A}^{AB}\big(T,\mathcal{U})\,,
\end{equation}
which consist of all profiles that can be attained at a fixed time $T>0$
 by $AB$-entropy solutions of~\eqref{eq1},\eqref{flux} with initial
data that varies inside  $\mathcal{U}$,
or by  $AB$-entropy solutions of~\eqref{eq1},\eqref{flux} with initial
data in $\mathcal{U}$
and connections $(A,B)\in\mathscr{C}$.
In the case where $\mathcal{U}$ is the whole space $\mathbf{L}^\infty(\mathbb{R})$,
we set 
\begin{equation}
\label{att-set-def2}
\mathcal{A}^{AB}(T)\doteq \mathcal{A}^{AB}\big(T,\mathbf{L}^\infty(\mathbb{R})\big),
\qquad
\mathcal{A}\big(T\big)\doteq
\mathcal{A}\big(T,\mathbf{L}^\infty(\mathbb{R}),\mathscr{C}_f\big)\,.
\end{equation}
We will provide a characterisation of the sets~\eqref{att-set-def2}
in terms of certain Ole\v{\i}nik-type type estimates on the decay
of positive waves, and we will establish the $\mathbf{L}^1$-compactness
of~\eqref{att-set-def1} for classes $\mathcal{U}$
of initial data with  values in compact convex sets, and for compact sets $\mathscr{C}$ of connections.

\section{Statement of the main results}
\label{sec:main results}
We present here the main results of the paper whose proof will be established in Sections~\ref{Teo1}, \ref{Teo2}. 
Throughout the following
\vspace{-8pt}
\begin{equation}
D^- \omega(x)= \liminf_{h\rightarrow0}   \frac{\omega(x+h)-\omega(x)}{h} ,\quad \quad 
D^+ \omega(x)= \limsup_{h\rightarrow0}   \frac{\omega(x+h)-\omega(x)}{h},
\end{equation}
will denote, respectively, the lower and
the  upper Dini derivative of a function $\omega$ at $x$.
We shall also use the notations
$f_{l,-}^{-1}\doteq
({f_l}_{\mid (-\infty,\theta_l]})^{-1}$,\,
$f_{r,-}^{-1}\doteq
({f_r}_{\mid (-\infty,\theta_r]})^{-1}$, for the inverse of the restriction of 
$f_l$, $f_r$  to their decreasing part, respectively, 
and 
$f_{l,+}^{-1}\doteq ({f_l}_{\mid [\theta_l,+\infty)})^{-1}$, \,
$f_{r,+}^{-1}\doteq ({f_r}_{\mid [\theta_{r},+\infty)})^{-1}$, for the inverse of the restriction of 
$f_l$, $f_r$ to their increasing part, respectively.
Then, we set 
\begin{equation}
\label{pimap-def}
\pi_{l,\pm}\doteq f_{l,\pm}^{-1} \circ f_l\,,
\qquad\quad
\pi_{r,\pm}\doteq f_{r,\pm}^{-1} \circ f_r\,,
\qquad\quad
\pi_{l,\pm}^{r}\doteq f_{l,\pm}^{-1} \circ f_r\,,
\qquad\quad
\pi_{r,\pm}^{l}\doteq f_{r,\pm}^{-1} \circ f_l\,.
\end{equation}
and 
we introduce the following sets that characterize the
left and right traces of an $AB$-entropy solution at the \linebreak flux-discontinuity interface
(see Remark~\ref{charact-lrtraces-admissible}):
\begin{equation}
\label{lrtraces-admissible-def}
\begin{aligned}
\mathcal{T}_1&\doteq
\big\{(u_l,u_r)\in (\theta_l,+\infty)\times(\theta_r,+\infty);\ u_l\geq\pi_{l,+}(A),\ B\leq u_r\leq \pi_{r,+}^l(u_l)\big\}\,,
\\
\noalign{\smallskip}
\mathcal{T}_2&\doteq
\big\{(u_l,u_r)\in (-\infty,\theta_l)\times(-\infty,\theta_r);\ \pi_{l,-}^r(u_r)\leq u_l\leq A,\ u_r\leq\pi_{r,-}(B)\big\}\,,
\\
\noalign{\smallskip}
\mathcal{T}_{3,-}&\doteq
\big\{(u_l,u_r)\in [\theta_l,+\infty)\times(-\infty,\theta_r);\ \pi_{l,+}(A)\leq u_l\leq \pi_{l,+}^r(u_r),\  u_r\leq \pi_{r,-}(B)\big\}\,,
\\
\noalign{\smallskip}
\mathcal{T}_{3,+}&\doteq
\big\{(u_l,u_r)\in (\theta_l,+\infty)\times(-\infty,\theta_r];\ u_l\geq\pi_{l,+}(A),\ \pi_{r,-}^l(u_l)\leq u_r\leq \pi_{r,-}(B)\big\}\,.
\end{aligned}
\end{equation}
%
\begin{thm}\label{Teo1}
Let $f$ be a flux as in~\eqref{flux} satisfying the assumptions {\bf (H1), (H2), (H3)},
and let $(A,B)\in\mathscr{C}_f$. Then, for any fixed $\text{T}>0$,
the set $\mathcal{A}^{AB}(T)$ in~\eqref{att-set-def2}
is given by 
\begin{equation}
\label{att-set-union-3}
\mathcal{A}^{AB}(T)= \mathcal{A}_1(T)\cup \mathcal{A}_2(T)\cup \mathcal{A}^{AB}_3(T)\,,
\end{equation}
where $\mathcal{A}_1(T), \mathcal{A}_2(T), \mathcal{A}^{AB}_3(T)$ are sets of 
functions $\omega\in\mathbf{L}^\infty(\mathbb{R})$ having  essential left and right limits at $x=0$,
defined as follows. 

 \begin{itemize}
  \item[$\mathcal{A}_{1}\text{(T)}$] is the set  of   all functions $\omega$
  that satisfy\,
  $\big(\omega(0-),\, \omega(0+)\big)\in \mathcal{T}_1$\,,
  and for which there exists  $R>0$ such that the following conditions hold.
 \begin{equation}
\begin{aligned}
&\omega(x) \geq
(f'_l)^{-1}\big({x}/{T}+f'_l(\omega(0-))\big)  \  \ \forall~x\in (-\infty,0),
\qquad\quad
\omega(x)\geq (f'_{r})^{-1}\big({x}/{T}\big)  \ \  \forall~x\in (0,R),
\\
\noalign{\medskip}
 &\omega(x) < (f'_{r})^{-1}\big({x}/{T}\big) \  \ \forall~x\in(R,+\infty),
 \qquad\qquad\qquad\qquad \omega(R-)\geq \omega(R+),
 \end{aligned}
 \label{c1}  
 \end{equation}
 \begin{align}
 &\hspace{-0.5in}
 D^+\omega(x)\leq 
 \begin{cases}
  &\hspace{-0.15in}1/\big({f''_{l}(\omega(x))\cdot T}\big)  \hspace{0.7in}\forall~x\in (-\infty,0)\,,\\
  \noalign{\smallskip}
  & \hspace{-0.15in}\dfrac{f'_r(\omega(x))\big[  f'_l \circ f_{l,+}^{-1}\circ f_r(\omega(x)) \big]^2 }
     {  \big[f''_l\circ f_{l,+}^{-1}\circ f_r(\omega(x))\big]\left[f'_r(\omega(x))\right]^2\big( f'_r(\omega(x))T-x \big)+x\big[f'_l\circ f_{l,+}^{-1}\circ f_r(\omega(x))\big] ^2  f''_r(\omega(x)) } \qquad \forall~x\in (0, R),
    \\
    \noalign{\medskip}
    &\hspace{-0.15in}{1}/\big({f''_{r}(\omega(x))\cdot T}\big) \hspace{0.7in}\forall~x\in (R,+\infty)\,.\\
  \end{cases}
  \label{bound totale1} 
 \end{align}

  \item[$\mathcal{A}_{2}\text{(T)}$]is the set  of   all  functions $\omega$
  that satisfy\,
  $\big(\omega(0-),\, \omega(0+)\big)\in \mathcal{T}_2$\,,
  and for which there exists  $L<0$ such that the following conditions hold.
   \begin{equation}
\begin{aligned}
&\omega(x) > (f'_{l})^{-1}\big({x}/{T}\big) \ \ \forall~x\in(-\infty, L),
\qquad\qquad  \qquad\qquad
\omega(x) \leq 
(f'_{l})^{-1}\big({x}/{T}\big)  \  \ \forall~x\in (L,0),
\\
\noalign{\medskip}
&\omega(x) \leq  (f'_{r})^{-1}\big({x}/{T}+f'_r(\omega(0+))\big)  \  \ \forall~x\in (0,+\infty),
\qquad\ \ \omega(L-)\geq \omega(L+),
\end{aligned}
\label{c4}  
 \end{equation}
 \begin{align}
 &\hspace{-0.5in}
 D^+\omega(x)\leq 
 \begin{cases}
  &\hspace{-0.15in}{1}/\big({f''_{l}(\omega(x))\cdot T}\big)  \hspace{0.7in}\forall~x\in (-\infty,L)\,,\\
  \noalign{\smallskip}
  & \hspace{-0.15in}\dfrac{f'_l(\omega(x))\left[  f'_r \circ f_{r,-}^{-1}\circ f_l(\omega(x)) \right]^2 }
     {  \left[f''_r\circ f_{r,-}^{-1}\circ f_l(\omega(x))\right]\left[f'_l(\omega(x))\right]^2\big( f'_l(\omega(x))T-x \big)+x\left[f'_r\circ f_{r,-}^{-1}\circ f_l(\omega(x))\right] ^2  f''_l(\omega(x)) } \qquad \forall~x\in (L,0),
    \\
    \noalign{\medskip}
    &\hspace{-0.15in}\big({1}/{f''_{r}(\omega(x))\cdot T}\big) \hspace{0.7in}\forall~x\in (0,+\infty)\,.\\
  \end{cases}
  \label{bound totale2} 
 \end{align}
  
  \vspace{-10pt}
  \item[$\mathcal{A}^{AB}_{3}\text{(T)}$] is the set  of all functions $\omega$
    for which there exist  $L\leq 0 \leq R$, such that 
    \vspace{-4pt}
  \begin{equation}
  \label{T3-cond1}
  \big(\omega(0-),\, \omega(0+)\big)\in
  \begin{cases}
   \mathcal{T}_{3,-}\cup \mathcal{T}_{3,+}\ \ &\text{if} \qquad L=R=0\,,
   \\
   \noalign{\smallskip}
   \big\{(A,B)\big\}\ \ &\text{if} \qquad L\leq 0 \leq R\,,
   \end{cases}
  \end{equation}
   
   \vspace{-10pt}
   \noindent
   and the following conditions hold.
   \begin{equation}
    \begin{aligned}
    &\omega(x)=A\quad \forall~x\in (L,0),\qquad\quad  \omega(x)=B\quad\forall~x\in (0,R),
    \\
    \noalign{\smallskip}
   &\omega(L-)\geq \omega(L+)\,,\qquad\quad 
   \omega(R-)\geq \omega(R+)\,,
    \\
   \noalign{\smallskip}
\allowdisplaybreaks
     &\omega(x) \geq 
     \begin{cases}
     (f'_{l})^{-1}\big({x}/{T}\big)\ \ &\text{if}\ \ L<0,
     \\
     \noalign{\smallskip}
     (f'_{l})^{-1}\big({x}/{T}\!+\!f'_l(\omega(0-))\big)  &\text{if} \ \ L=0,
     \end{cases}
     \qquad \forall~x\in (-\infty, L),
 \\
    \noalign{\smallskip}
     &\omega(x) \leq 
     \begin{cases}
     (f'_{r})^{-1}\big({x}/{T}\big)\ \ &\text{if}\ \ R>0,
     \\
     \noalign{\smallskip}
     (f'_{r})^{-1}\big({x}/{T}\!+\!f'_r(\omega(0+))\big)  &\text{if} \ \ R=0,
     \end{cases}
      \qquad \forall~x\in (R,+\infty),
     \end{aligned}
     \label{c7}
     \end{equation}
\vspace{-15pt}
\begin{align}
 D^+\omega(x)&\leq 
 \begin{cases}
  {1}/\big({f''_{l}(\omega(x))\cdot T}\big) \quad & \forall~x\in (-\infty, L)\,,\\
  \noalign{\smallskip}
  \big({1}/{f''_{r}(\omega(x))\cdot T}\big) \quad &\forall~x\in (R,+\infty)\,.\\
  \end{cases} 
   \label{bound totale3} 
 \end{align}
 
\end{itemize}
\end{thm}

\begin{figure}[!htbp] 
\begin{center}
\begin{tikzpicture}
[scale=1.4,cap=round]
 \def\costhirty{0.8660256}

\tikzstyle{axes}=[]
  \tikzstyle{important line}=[very thick]
  \tikzstyle{information text}=[rounded corners,fill=red!10,inner sep=1ex]

\draw[thick,->] (-2.2,0) -- (2.2,0) node[right] {$x$};
\draw[thick,->] (0,0) -- (0,1.9) node[right] {$t$};

\draw [domain=-2.2:2.2] plot (\x, {1.6});
\draw [dotted] (0,0.4) -- (0.9,1.6);
\draw [dotted] (0.4,0) -- (0.9,1.6);
\draw [dotted] (0,0.9) -- (0.6,1.6);
\draw [dotted] (0,1.25) -- (0.17,1.6); 
\draw [dotted] (0,1.25) -- (-0.9,0);
\draw [dotted] (0,0.9) -- (-0.7,0);
\draw [dotted] (-0.2,0) -- (0,0.4);
\draw (0.9,1.6) node[anchor=south] {$R$};

\draw [dotted] (1.2,1.6) -- (1.3, 0);
\draw [dotted] (1.8,1.6) -- (1.6,0);
\draw [dotted] (-0.8,1.6) -- (-1.1, 0);
\draw [dotted] (-1.6,1.6) -- (-1.4,0);
\draw (0,1.7) node[anchor=east] {$T$};
\draw (0.6,1.6) node[anchor=south] {$x$};
\draw (-0.7,0) node[anchor=north] {$\varphi_{1}(x)$};
\draw (1.8,1.6) node[anchor=south] {$y$};
\draw (1.6,0) node[anchor=north] {$\varphi_{1}(y)$};
\draw (-1.6,1.6) node[anchor=south] {$z$};
\draw (-1.5,0) node[anchor=north] {$\varphi_{1}(z)$};
\end{tikzpicture}
\end{center}
\vspace{-10pt}
\caption{Characteristics's behavior for profiles in  $\mathcal{A}_{1}(T)$
\label{figure 2-1}}
\vspace{8pt}
%
\begin{center}
\begin{tikzpicture}
[scale=1.5,cap=round]
 \def\costhirty{0.8660256}

\tikzstyle{axes}=[]
  \tikzstyle{important line}=[very thick]
  \tikzstyle{information text}=[rounded corners,fill=red!10,inner sep=1ex]

\draw[thick,->] (-2.2,0) -- (2.2,0) node[right] {$x$};
\draw[thick,->] (0,0) -- (0,1.9) node[right] {$t$};

\draw [domain=-2.2:2.2] plot (\x, {1.6});
\draw [dotted] (0,0.4) -- (-0.9,1.6);
\draw [dotted] (-0.4,0) -- (-0.9,1.6);
\draw [dotted] (0,0.9) -- (-0.6,1.6);
\draw [dotted] (0,1.25) -- (-0.17,1.6); 
\draw [dotted] (0,1.25) -- (0.9,0);
\draw [dotted] (0,0.9) -- (0.7,0);
\draw [dotted] (0.2,0) -- (0,0.4);
\draw (-0.9,1.6) node[anchor=south] {$L$};

\draw [dotted] (1.2,1.6) -- (1.3, 0);
\draw [dotted] (1.8,1.6) -- (1.6,0);
\draw [dotted] (-1.3,1.6) -- (-0.9, 0);
\draw [dotted] (-1.6,1.6) -- (-1.4,0);
\draw [dotted] (0.6,1.6) -- (1.1,0);
\draw (0,1.7) node[anchor=east] {$T$};
\draw (-1.3,1.6) node[anchor=south] {$x$};
\draw (-0.7,0) node[anchor=north] {$\varphi_{2}(x)$};
\draw (1.8,1.6) node[anchor=south] {$y$};
\draw (1.6,0) node[anchor=north] {$\varphi_{2}(y)$};
\draw (-1.6,1.6) node[anchor=south] {$z$};
\draw (-1.5,0) node[anchor=north] {$\varphi_{2}(z)$};
\end{tikzpicture}
\end{center}
\vspace{-10pt}
\caption{Characteristics's behavior for profiles in  $\mathcal{A}_{2}(T)$
\label{figure  2-2}}
\vspace{8pt}
\begin{center}
\begin{tikzpicture}
[scale=1.5,cap=round]
 \def\costhirty{0.8660256}

\tikzstyle{axes}=[]
  \tikzstyle{important line}=[very thick]
  \tikzstyle{information text}=[rounded corners,fill=red!10,inner sep=1ex]

\draw[thick,->] (-2.2,0) -- (2.2,0) node[right] {$x$};
\draw[thick,->] (0,0) -- (0,1.9) node[right] {$t$};

\draw [domain=-2.2:2.2] plot (\x, {1.6});
\draw [dotted] (-0.4,0) -- (-0.9,1.6);
\draw [dotted] (0,0) -- (-0.9,1.6);
\draw [dotted] (0,0) -- (0.5,1.6);

\draw (-0.9,1.6) node[anchor=south] {$L$};
\draw (0.5,1.6) node[anchor=south] {$R$};
\draw (-0.2,0.9) node[anchor=south] {$A$};
\draw (0.2,0.9) node[anchor=south] {$B$};

\draw [dotted] (1.2,1.6) -- (1.3, 0);
\draw [dotted] (1.8,1.6) -- (1.6,0);
\draw [dotted] (-1.3,1.6) -- (-0.9, 0);
\draw [dotted] (-1.6,1.6) -- (-1.4,0);
\draw [dotted] (0.6,1.6) -- (1.1,0);
\draw (0,1.7) node[anchor=east] {$T$};
\draw (-1.3,1.6) node[anchor=south] {$x$};
\draw (-0.7,0) node[anchor=north] {$\varphi_{3}(x)$};
\draw (1.8,1.6) node[anchor=south] {$y$};
\draw (1.6,0) node[anchor=north] {$\varphi_{3}(y)$};
\draw (-1.6,1.6) node[anchor=south] {$z$};
\draw (-1.5,0) node[anchor=north] {$\varphi_{3}(z)$};
\end{tikzpicture}
\end{center}
\vspace{-10pt}
\caption{Characteristics's behavior for profiles in  $\mathcal{A}^{AB}_{3}(T)$
\label{figure  2-3}}

%
\bigskip
\noindent
\setcounter{rem}{3}
\begin{rem}
\label{main-thm-interpret2}
The conditions~\eqref{bound totale1}, 
\eqref{bound totale2}, \eqref{bound totale3}
reflect the fact that, since the fluxes are strictly convex, positive waves of 
$AB$-entropy solutions
decay in time. 
Such conditions are  sufficient to 
guarantee the exact-time controllability of~\eqref{eq1}-\eqref{datum}. 
In fact, starting at a time $T>0$ with a profile 
$\omega\in \mathcal{A}_1(T)\cup \mathcal{A}_2(T)\cup \mathcal{A}^{AB}_3(T)$, 
because of~\eqref{bound totale1}, 
\eqref{bound totale2}, \eqref{bound totale3} one can trace 
backward the
 (generalized) characteristics $\xi_1$, $\xi_2$ through points  $x_1<x_2$ 
 without crossing  in $\mathbb{R}\times (0,T)$,
 unless $\omega(x_1)=A$ and $\omega(x_2)=B$, in which cases the characteristics $\xi_1, \xi_2$ intersects
 only at $x=0$ (see  Figures~\ref{figure 2-1}-\ref{figure 2-3}).
In particular, by~\eqref{c1}, \eqref{c4}, the inequalities~\eqref{bound totale1}, \eqref{bound totale2}
imply that
\begin{equation}
\label{bound totale 4}
D^+ \omega(x) < \frac{f'_l(\omega(x))}{x f''_l(\omega(x))} 
\quad\forall~x\in(L,0)\,,
\qquad\quad
D^+ \omega(x) < \frac{f'_r(\omega(x))}{x f''_r(\omega(x))} 
\quad\forall~x\in(0,R)\,,
\end{equation}
and we recover the same type of boundary controllability condition 
derived in~\cite{MR1616586},
if we regard the left and right traces at $x=0$ as controls.
Notice that  we have in~\eqref{bound totale 4}
a strict inequality since here, differently from~\cite{MR1616586}, characteristics having slope with the same sign 
cannot intersect even at $x=0$ (they can intersect only at $t=0$).
\end{rem}
\smallskip
The characterization of the attainable set $\mathcal{A}^{AB}(T)$ provided by Theorem~\ref{Teo1} 
yields  the ${\bf L^1}$-compactness of the attainable sets $\mathcal{A}^{AB}\big(T,\mathcal{U}\big)$,
$\mathcal{A}\big(T,\mathcal{U},\mathscr{C}\big)$ 
in~\eqref{att-set-def1}
for classes $\mathcal{U}$ of admissible controls uniformly bounded and with values in convex and closed sets,
as stated in the following
\begin{thm}\label{Teo2}
Let $G:\mathbb{R}\hookrightarrow \mathbb{R}$ be a measurable, 
bounded multifunction with convex and closed values, and let $\mathscr{C}\subset \mathscr{C}_f$ be a compact set of connections.
Consider the set
\vspace{-5pt}
\begin{equation}\label{insieme controlli}
\mathcal{U}=\left\lbrace \overline{u}\in {\bf L^{\infty}}(\mathbb{R}):\overline{u}(x)\in G(x)\text{ for }a.e.\ \ x\in \mathbb{R} \right\rbrace. 
\end{equation}

\vspace{-5pt}
\noindent
 Then, under the same assumptions of Theorem~\ref{Teo1}, 
 for any fixed $T>0$, 
 the sets $\mathcal{A}^{AB}\big(T,\mathcal{U}\big)$,
$\mathcal{A}\big(T,\mathcal{U},\mathscr{C}\big)$ 
in~\eqref{att-set-def1}
 are compact in the ${\bf L^1}_{loc}(\mathbb{R})$-topology and,
 letting $\mathcal{S}^{AB}_{(\cdot)} \overline u\,\big|_{T}$ denote the restriction of $\mathcal{S}^{AB}_{(\cdot)} \overline u$
 to $\mathbb{R}\times [0,T]$,
 the sets 
 \vspace{-2pt}
\begin{equation}
\label{att-set-def3}
\mathcal{A}^{AB}\big(\mathcal{U}\big)\doteq  
\big\{\mathcal{S}^{AB}_{(\cdot)} \overline u \,\big|_T: \overline u\in\mathcal{U}\big\}\,,\qquad
\mathcal{A}\big(\mathcal{U},\mathscr{C}\big)\doteq
\bigcup_{(A,B)\in\mathscr{C}} \mathcal{A}^{AB}\big(\mathcal{U})\,,
\end{equation}

\vspace{-9pt}
\noindent
are compact in the ${\bf L^1}_{loc}(\mathbb{R}\times [0,T])$-topology.
\end{thm}
\end{figure}

\noindent
\setcounter{rem}{2}
\begin{rem}
\label{main-thm-interpret1}
Notice that, by the strict convexity assumption (H1) on $f''_l, f''_r$,
and relying on~\eqref{c1}, \eqref{c4}, we deduce that the right-hand side of~\eqref{bound totale1}, 
\eqref{bound totale2}, \eqref{bound totale3} is always nonnegative, and it is bounded on
any set bounded away from $x=0$. Therefore, 
any $\omega\in  \mathcal{A}_1(T)\cup \mathcal{A}_2(T)\cup \mathcal{A}^{AB}_3(T)$  is an equivalence class of 
bounded measurable functions that haves finite total increasing variation 
(and hence finite total variation as well) on subsets of $\mathbb{R}$
bounded away from the origin.
Moreover, by assumption 
any \, $\omega\in  \mathcal{A}_1(T)\cup \mathcal{A}_2(T)\cup \mathcal{A}^{AB}_3(T)$ \,
admits one-sided
limits at $x=0$.
Hence, 
any element of $ \mathcal{A}_1(T)\cup \mathcal{A}_2(T)\cup \mathcal{A}^{AB}_3(T)$ admits one-sided limits at every point.
\end{rem}
\pagebreak

\medskip
In turn, the compactness of the attainable sets yields the existence of optimal solutions 
 for a class of minimization (maximization) problems, by considering a minimizing (maximizing)
 sequence for the corresponding cost functionals. 
\begin{cor}
\label{existence-opt-sol}
Let $G$ be multivalued map as in Theorem~\ref{Teo2} and assume that $G(x)=0$
for all $x\in\mathbb{R}\setminus K$, for some bounded set $K\subset\mathbb{R}$.
Given an  interval $I\subset\mathbb{R}$ and $T>0$,
let $F_1:L^1(I)\rightarrow \mathbb{R}$,  $F_2:L^1(I\times [0,T])\rightarrow \mathbb{R}$
be lower semicontinuous functionals, and let $\mathcal{U}$ be the set of admissible controls defined in~\eqref{insieme controlli}. Then, under the same assumptions of Theorem~\ref{Teo1},
the optimal control problems
\begin{equation}
\min_{\overline{u}\in\mathcal{U}}{F_1(S_T\overline{u}(\cdot))},
\qquad\quad
\min_{\overline{u}\in\mathcal{U}}{F_2(S_{(\cdot)}\overline{u}(\cdot))},
\end{equation}
admit a solution. 
If we assume that $F_1, F_2$ are upper semicontinuous functionals, then there exists a solution
of the maximization problems
\begin{equation}
\max_{\overline{u}\in\mathcal{U}}{F_1(S_T\overline{u}(\cdot))},
\qquad\quad
\max_{\overline{u}\in\mathcal{U}}{F_2(S_{(\cdot)}\overline{u}(\cdot))}.
\end{equation}
\end{cor}

\section{Structure of $AB$-entropy solutions and a technical lemma}
\label{sec-struct-sol}
We analyze here some structural properties of $AB$-entropy solutions
and we derive a technical lemma 
on the relation between upper bounds on the Dini derivative and the monotonicity of suitable maps, that will be useful for the proofs
of Theorem~\ref{Teo1} and Theorem~\ref{Teo2}.

\setcounter{rem}{4}
\begin{rem}
\label{propert3-AB-sol}
By the analysis in~\cite[Section 3.1]{MR2291816}  it follows that  the Riemann solver 
associated to a given connection~$(A,B)$ enjoys the following properties. 
Letting $u(x,t)$ be the $AB$-entropy solution 
of the Cauchy problem for~\eqref{eq1},\eqref{flux}, with initial data 
$\overline u(x)= u^-$ if $x<0$, and $\overline u(x)= u^+$ if $x>0$, 
for any given
$a\leq \theta_r, b\geq \theta_l$, $a<b$, there holds
\begin{equation}
\label{linf-bound-1}
\{A,B,u^-,u^+\}\subseteq [a,b]
\quad \Longrightarrow\quad 
u(x,t)\in \big[\min\{a,\pi_{l,-}^r (a)\},\ \max\{b,\pi_{r,+}^l(b)\}\big]
\qquad \forall~x\in\mathbb{R},\ t\geq 0\,.
\end{equation}
Moreover, if $A,B\in [0,1]$, 
by the assumptions {\bf H1), H2)} on $f_l, f_r$, one has
\begin{equation}
\label{pi-maps-prop2}
\{A,B,u^-,u^+\}\subseteq [0,1]\quad \Longrightarrow\quad u(x,t)\in [0,1]
\qquad \forall~x\in\mathbb{R},\ t\geq 0\,.
\end{equation}
Observe that, if $u(x,t)$ is a front tracking solution (cfr.~\cite[Section 4]{MR2291816}) constructed with approximate Riemann
solvers 
that satisfy~\eqref{linf-bound-1}, \eqref{pi-maps-prop2},
then the same type of a-priori bounds hold.
In fact, $u$ can assume values which do not belong to the range of the initial data
$\overline u$
only on regions adjacent to the discontinuity $x=0$
(from the left or from the right), and such values always belong to the interval
$$\big[\inf\big\{\min\{\overline u(x),\pi_{l,-}^r (\overline u(x))\};\, x\in\mathbb{R}\big\}, \
\sup\big\{\max\{\overline u(x),\pi_{r,+}^l (\overline u(x))\};\, x\in\mathbb{R}\big\}\big].$$
Hence, since a general solution of a Cauchy problems for \eqref{eq1},\eqref{flux}
can be obtained as limit of front tracking solutions (see~\cite{MR2291816,MR1158825}),
we deduce the following
a-priori bounds for any $u(x,t)\doteq \mathcal{S}^{AB}_t \overline u(x)$,
with $\overline u\in {\bf L^\infty}(\mathbb{R})$:
\begin{equation}
\label{linf-bound-3}
\{A,B\}\cup\{\overline u(x);\ x \in \mathbb{R}\}
\subseteq [a,b]
\quad \Longrightarrow\quad 
u(x,t)\in \big[\min\{a,\pi_{l,-}^r (a)\},\ \max\{b,\pi_{r,+}^l(b)\}\big]
\qquad \forall~x\in\mathbb{R},\ t\geq 0\,,
\end{equation}
and
\begin{equation}
\label{pi-maps-prop4}
\{A,B\}\cup\{\overline u(x);\ x \in \mathbb{R}\} \subseteq [0,1]\quad \Longrightarrow\quad u(x,t)\in [0,1]
\qquad \forall~x\in\mathbb{R},\ t\geq 0\,.
\end{equation}
Moreover, if the initial data $\overline u$ vanishes outside a bounded set $K$, then there will be some 
bounded set $K'$ such that $\text{supp}(u(\cdot,t)) \subset K'$ for all $t>0$.
\end{rem}

\medskip

The classical theory of generalized characteristics for conservation laws with continuous 
and convex flux~\cite{MR0457947}
guarantees that backward characteristics, lying in the same quarter of plane $(-\infty,0]\times [0,+\infty)$
or $[0, +\infty)\times [0,+\infty)$,  never intersect
at times $t>0$ in points $x\neq 0$.
A fundamental feature of $AB$-entropy solutions is that backward generalized characteristics
cannot intersect at times $t>0$ even along the discontinuity interface $x=0$, unless 
$(u_l(s), u_r(s))=(A,B)$ for all $0<s \leq t$. It follows in particular that no rarefaction fan can be originated 
at $x=0$ and $t>0$.
This property is the consequence of the next
Proposition. We recall that a generalized characteristic $\xi(t), t\in (t', t'')$ 
for a conservation law $u_t+f(u)_x=0$ is called genuine
if, for almost every $t\in (t', t'')$, there holds $u(\xi(t)-,t)=u(\xi(t)+,t)=v$ for
some constant $v$ such that $f'(v)=\dot\xi$.
Thus, genuine  characteristics are segments of lines 
which may intersect only at their end points~\cite{MR0457947}.
%
%

\begin{prop}\label{Norarefactions}
Let $f$ be a flux as in~\eqref{flux} satisfying the assumptions {\bf H1), H2), H3)},
and let $u(x,t)$ be an $AB$-entropy solution of~\eqref{eq1},\eqref{flux}-\eqref{datum},
for some initial data $\overline u\in {\bf L}^\infty(\mathbb{R})$ and  a connection $(A,B)\in\mathscr{C}_f$.
Then, at any time $\overline t>0$ the following hold.
\begin{itemize}
\item[(i)]
If  $u_l(\,\overline t+)< \theta_l$ and $u_r(\,\overline t+)> \theta_r$, then $(u_l(t\pm), u_r(t\pm))= (A,B)$
for all $t\in (0,t^*)$, for some $t^*>\overline t$. Moreover,
there exist exactly two forward, genuine, characteristics $\eta', \eta''$, starting at $(0, \overline t\,)$,
which lie in $(-\infty,0)\times (\,\overline t,  t^*)$ and $(0,+\infty)\times (\,\overline t,  t^*)$, respectively.

\item[(ii)]
If  $u_l(\,\overline t+)\geq  \theta_l$ or $u_r(\,\overline t+)\leq \theta_r$, then 
there exists at most a single forward, genuine, characteristic starting at $(0, \overline t\,)$
and lying in $(\mathbb{R}\setminus \{0\})\times (\,\overline t,  t^*)$,
for some $t^*>\overline t$.
\end{itemize}
\end{prop}

\begin{proof}
We shall distinguish three cases.\\
\vspace{-7pt}

\noindent
{\bf Case 1.}
{\it  $u_l(\,\overline t+)>\theta_l$ and $u_r(\,\overline t+)\neq  \theta_r$, \ or \ $u_r(\,\overline t+)<  \theta_r$
and $u_l(\,\overline t+)\neq \theta_l$.}\\
{\bf 1a)} If $u_l(\,\overline t+)>\theta_l$, $u_r(\,\overline t+)>  \theta_r$
and $u(0+,\overline t\,)\geq u_r(\,\overline t+)$ (see Figure~\ref{figure 3}), then 
consider two sequences of points $\{t_n, \, t_n \downarrow 0\}$,
and $\{(x_n,\overline t\,), \, x_n\downarrow 0\}$,
of continuity for $u_r$ and $u$, respectively.
Tracing the backward genuine characteristics
(with positive slopes)
through $(0,t_n)$ and $(x_n,\overline t\,)$
one  deduces that there exist sequences of points 
$\{(x'_n,\overline t\,), \, x'_n\uparrow 0\}$, and $\{t'_n, \, t'_n \uparrow 0\}$,
such that $u(x'_n,\overline t\,)\to u_l(\,\overline t+)$ and $u_r(t'_n)\to u(0+,\overline t\,)$.
Hence, there holds
$u(0-,\overline t\,)=u_l(\,\overline t+)$, $u_r(\,\overline t-)=u(0+,\overline t\,)$.
Now observe that, if $u_l(\,\overline t-)\neq u(0-,\overline t\,)$,
then there should be
a shock with positive slope 
arriving in $(0,\overline t\,)$ (or generated in $(0,\overline t\,)$) and
connecting the left state~$u(0-,\overline t\,)$ 
with the right state $u_l(\,\overline t-)$.
Such a shock is entropy admissible for the conservation law with flux~$f_l$
and has positive slope if and only if $u_l(\,\overline t-)< u(0-,\overline t\,)$
and $f_l(u_l(\,\overline t-))<f_l(u(0-,\overline t\,))$.
Since by~\eqref{int-entr-cond3} 
one has $f_l(u_l(\,\overline t-))=f_r(u_r(\,\overline t-))$,
and because of $u(0-,\overline t\,)=u_l(\,\overline t+)$,
from $f_l(u_l(\,\overline t-))<f_l(u(0-,\overline t\,))$
it follows that $f_r(u_r(\,\overline t-))<f_l(u_l(\,\overline t+))$.
On the other hand, $\theta_r<u_r(\,\overline t+)\leq u(0+,\overline t\,)=u_r(\,\overline t-)$
implies $f_r(u_r(\,\overline t+))\leq f_r(u_r(\,\overline t-))$
which is in contrast with $f_r(u_r(\,\overline t-))<f_l(u_l(\,\overline t+))$.
Therefore, $u_l(\,\overline t+)>\theta_l$, $u_r(\,\overline t+)>  \theta_r$
and $u(0+,\overline t)\geq u_r(\,\overline t+)$,
together imply that $u_l(\,\overline t-)= u(0-,\overline t\,)=u_l(\,\overline t+)$.
Moreover, since by~\eqref{int-entr-cond3}
one has $f_l(u_l(\,\overline t+))=f_r(u_r(\,\overline t+))$,
from $\theta_r<u_r(\,\overline t+)\leq u(0+,\overline t\,)=u_r(\,\overline t-)$
it follows that $f_r(u_r(\,\overline t+))\leq f_r(u_r(\,\overline t-))\leq 
f_r(u_r(\,\overline t+))$. Hence, 
$u_l(\,\overline t+)>\theta_l$, $u(0+,\overline t\,)\geq u_r(\,\overline t+)>  \theta_r$
together imply that $u_l(\,\overline t\pm)= u(0-,\overline t\,)$
and  $u_r(\,\overline t\pm)= u(0+,\overline t\,)$,
which shows that from $(0,\overline t\,)$
it emerges a single forward genuine characteristic, lying on 
$(0,+\infty)\times (\,\overline t,  t^*)$, for some $t^*>\overline t$,
and property (ii) is verified.
\\
{\bf 1b)} If $u_l(\,\overline t+)>\theta_l$
and $\theta_r\leq u(0+,\overline t\,)< u_r(\,\overline t+)$, then  there is
a shock with positive slope 
starting at $(0,\overline t\,)$  and
connecting the left state $u_r(\,\overline t+)$ 
with the right state $u(0+,\overline t\,)$. Moreover, tracing 
the backward genuine characteristics
through a sequence of points $(x_n,\overline t\,), \, x_n \uparrow 0$,
of continuity for $u$,
one  deduces that $u_r(\,\overline t-)=u(0+,\overline t\,)$.
Hence, by~\eqref{int-entr-cond3} 
one has $f_l(u_l(\,\overline t-))=f_r(u_r(\,\overline t-))<f_r(u_r(\,\overline t+))=
f_l(u_l(\,\overline t+))$.
On the other hand, by the observations in case {\bf 1a)} it follows that 
$u(0-,\overline t\,)=u_l(\,\overline t+)$, which implies $f_l(u(0-,\overline t\,))>f_l(u_l(\,\overline t-))$.
Thus, it must be $u(0-,\overline t\,)>u_l(\,\overline t-)$, 
and there is a shock with positive slope arriving at $(0,\overline t\,)$
(or generated in~$(0,\overline t\,)$) connecting the left state $u(0-,\overline t\,)$
with the right state $u_l(\,\overline t-)\in\{\pi_{l,-}^r(u(0+,\overline t\,)),\pi_{l,+}^r(u(0+,\overline t\,))\}$.
Therefore, if $u_l(\,\overline t+)>\theta_l$
and $\theta_r\leq u(0+,\overline t\,)< u_r(\,\overline t+)$, then there is no forward, genuine characteristic, 
emerging from~$(0, \overline t\,)$, there is a single (forward) shock starting at $(0, \overline t\,)$
with positive slope, and property (ii) is verified.
\\
{\bf 1c)} If $u_l(\,\overline t+)>\theta_l$
and $u(0+,\overline t\,)<\theta_r<u_r(\,\overline t+)$, then
with similar arguments to case {\bf 1b)} one deduces that: \\
- there is 
a shock with positive slope 
starting at $(0,\overline t\,)$  and
connecting the left state $u_r(\,\overline t+)$ 
with the right state $u(0+,\overline t)>\pi_{r,-}(u_r(\,\overline t+))$;\\
- there is 
a shock with positive slope arriving at $(0,\overline t\,)$
(or generated in $(0,\overline t\,)$) connecting the left state $u(0-,\overline t\,)=u_l(\,\overline t+)$
with the right state $u_l(\,\overline t-)\in\{\pi_{l,-}^r(u_r(\,\overline t-)),\pi_{l,+}^r(u_r(\,\overline t-))\}$;\\
- either $u_r(\,\overline t-)=u(0+,\overline t\,)$, or $u_r(\,\overline t-)>u(0+,\overline t\,)$,
and in this latter case there is a shock with negative slope arriving at $(0,\overline t\,)$
(or generated in $(0,\overline t\,)$) that connects the left state $u_r(\,\overline t-)\in(u(0+,\overline t\,), \pi_{r,+}(u(0+,\overline t\,)))$
with the right state $u(0+,\overline t\,)$.
\\
Therefore, if $u_l(\,\overline t+)>\theta_l$
and $u(0+,\overline t\,)<\theta_r<u_r(\,\overline t+)$, then 
as in case {\bf 1b)} there is no forward, genuine, characteristic
emerging from $(0, \overline t\,)$, while there is a single (forward) shock starting at $(0, \overline t\,)$,
which has negative slope. Hence, property (ii) is verified.\\
{\bf 1d)} If $u_l(\,\overline t+)<\theta_l$ and $u_r(\,\overline t+)<  \theta_r$, then 
we can proceed as in cases {\bf 1a)-1c)}
to conclude that:
either $u_l(\,\overline t\pm)= u(0-,\overline t\,)$,  $u_r(\,\overline t\pm)= u(0+,\overline t\,)$,
and
it emerges a single forward genuine characteristic, lying on 
$(-\infty,0)\times (\,\overline t,  t^*)$, for some $t^*>\overline t$,
or $u_l(\,\overline t+)< u(0-,\overline t\,)$, $u_r(\,\overline t+)= u(0+,\overline t\,)$,
and there is no forward, genuine characteristic, 
emerging from $(0, \overline t\,)$, while there is a single (forward) shock starting at $(0, \overline t\,)$,
which has  negative slope.
Thus, property (ii) is verified.\\
{\bf 1e)} If $u_l(\,\overline t+)>\theta_l$ and $u_r(\,\overline t+)<  \theta_r$, then
with the same arguments as above we deduce that $u(0-,\overline t\,)=u_l(\,\overline t+)$,
$u(0+,\overline t\,)=u_r(\,\overline t+)$, and
by~\eqref{int-entr-cond3}
one of the following subcases occurs:\\
- $u_l(\,\overline t+)=u_l(\,\overline t-)$, $u_r(\,\overline t+)=u_r(\,\overline t-)$, and 
in a neighbourhood of $\overline t$ the characteristics 
are crossing the line $x=0$ with positive slopes on the left side, 
with negative slopes on the right side;\\
- $u_l(\,\overline t-)\leq \theta_l<u_l(\,\overline t+)$, $u_r(\,\overline t+)<u_r(\,\overline t-)\leq \theta_r$,
there is a shock with positive slope arriving
at $(\,\overline t,0)$ (or generated in $(0,\overline t\,)$),
which connects the left state $u(0-,\overline t\,)=u_l(\,\overline t+)$
with the right state $u_l(\,\overline t-)$, there is a shock with negative slope connecting the
left state $u_r(\,\overline t-)=\pi_{r,-}^l(u_l(\,\overline t-))$
with the right state $u(0+,\overline t\,)=u_r(\,\overline t+)$,
and in a left neighbourhood of $\overline t$ the characteristics 
are crossing the line $x=0$ with negative slopes on both sides;\\ 
- $\theta_l<u_l(\,\overline t-)=A<u_l(\,\overline t+)$, $u_r(\,\overline t+)<u_r(\,\overline t-)=B<\theta_r$,
there are two shocks with positive and negative slopes arriving at $(0,\overline t\,)$
as in the previous case,
and in a left neighbourhood of $\overline t$ the characteristics 
are crossing the line $x=0$ with  with positive slopes on the left side, 
with negative slopes on the right side;\\
- $\theta_l<u_l(\,\overline t-)<u_l(\,\overline t+)$, $u_r(\,\overline t+)<\theta_r\leq u_r(\,\overline t-)$,
there are two shocks with positive and negative slopes arriving at $(0,\overline t\,)$
as in the previous case,
and in a left neighbourhood of $\overline t$ the characteristics 
are crossing the line $x=0$  with positive slopes on both sides.\\
In all subcases of {\bf 1e)} there is no forward characteristic emerging from $(0, \overline t\,)$
and hence  property (ii) is verified.
\vspace{-5pt}
\begin{figure}[!htbp] \label{figure 3}
\begin{center}

\begin{tikzpicture}[scale=1.2,cap=round]
\title{case 1}
 \def\costhirty{0.8660256}

\tikzstyle{axes}=[]
  \tikzstyle{important line}=[very thick]
  \tikzstyle{information text}=[rounded corners,fill=red!10,inner sep=1ex]

\draw[thick,->] (-2.2,-0.2) -- (2.2,-0.2) node[right] {$x$};
\draw[thick,->] (0,-0.2) -- (0,2) node[right] {$t$};

\draw [domain=-2.2:2.2] plot (\x, {0.85});

\draw [thick] (0,0.85) -- (1.1,1.9);
\draw [dashed] (0,0.85) -- (1,1.9);
\draw [dashed] (0,0.85) -- (0.9,1.9);
\draw [dashed] (0,0.85) -- (0.8,1.9);
\draw [dashed] (0,0.85) -- (0.7,1.9);
\draw [thick] (0,0.85) -- (0.6,1.9); 

\draw (0,0.99) node[anchor=east] {$\bar{t}$};
\draw (0.9,1.58) -- (0.9,1.62);

\draw [thick, blue] (1.2,1.9) -- (0,0.75) -- (- 0.55,-0.2);

\draw [thick, red] (0.5,1.9) -- (0,0.97) -- (-0.25,-0.2); 

\end{tikzpicture}
\quad
\begin{tikzpicture}[scale=1.2,cap=round]
 \def\costhirty{0.8660256}

\tikzstyle{axes}=[]
  \tikzstyle{important line}=[very thick]
  \tikzstyle{information text}=[rounded corners,fill=red!10,inner sep=1ex]

\draw[thick,->] (-2.2,-0.2) -- (2.2,-0.2) node[right] {$x$};
\draw[thick,->] (0,-0.2) -- (0,2) node[right] {$t$};

\draw [domain=-2.2:2.2] plot (\x, {0.85});

\draw [thick] (0,0.85) -- (1.1,1.9);
\draw [dashed] (0,0.85) -- (1,1.9);
\draw [dashed] (0,0.85) -- (0.9,1.9);
\draw [dashed] (0,0.85) -- (0.8,1.9);
\draw [dashed] (0,0.85) -- (0.7,1.9);
\draw [thick] (0,0.85) -- (0.6,1.9); 

\draw (0,0.99) node[anchor=east] {$\bar{t}$};
\draw (0.9,1.58) -- (0.9,1.62);

\draw [thick, blue] (0,0.75) -- (- 0.55,-0.2);
\draw [thick,red] (-0.25,0.85) -- (-0.1,-0.2);

\end{tikzpicture}
\caption{On the left \text{case 1a}, on the right \text{case 2a}}
\end{center}
\end{figure}
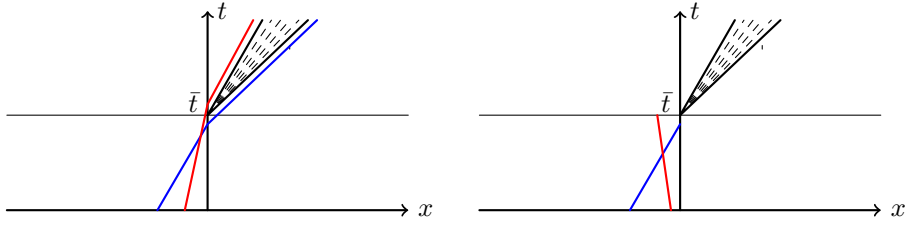

\vspace{-5pt}

\noindent
{\bf Case 2.}
{\it  $u_l(\,\overline t+)<\theta_l$ \ and \ $u_r(\,\overline t+)>  \theta_r$.}\\
Let $t^*>\overline t$ be such that $u_l(t)< \theta_l$ and $u_r(t)> \theta_r$
for all $t\in (\overline t, t^*)$.
Then, by~\eqref{int-entr-cond3}
this implies that $(u_l(t), u_r(t))=(A,B)$
for all $t\in (\overline t, t^*)$, with $A<\theta_l$, $B>\theta_r$,
and hence there holds $(u_l(\,\overline t+), u_r(\,\overline t+)) = (A,B)$.
\\
{\bf 2a)} If $u(0-,\overline t\,)<A$, then 
tracing the backward genuine characteristics
(with negative slopes)
through a sequence of points $(x_n,\overline t\,), \, x_n \uparrow 0$,
of continuity for $u$,
one  deduces that
$u_l(\overline t-)=u(0-,\overline t\,)$.
Hence, $u_l(\overline t-)<A$ and $f_l(u_l(\overline t-))>f(A)$.
By~\eqref{int-entr-cond3}
this implies that 
$u_r(\overline t-)<\pi_{r,-}(B)$.
Observe that if $u(0+,\overline t\,)< B$, then
a shock with positive slope should emerge from $(0,\overline t\,)$,
with left state $u_r(\,\overline t+)=B$ and right state $u(0+,\overline t\,)$.
But, this implies that $u(0+,\overline t\,)>\pi_{r,-}(B)$.
On the other hand, from $u_r(\overline t-)\neq u(0+,\overline t\,)$
it follows that there should be
a shock with negative slope 
arriving in $(0,\overline t\,)$ (or generated in $(0,\overline t\,)$) and
connecting the left state $u_r(\,\overline t-)$ 
with the right state $u(0+,\overline t\,)>u_r(\,\overline t-)$,
which is not entropy admissible for the conservation law with flux $f_r$.
Therefore, if 
$u_l(\,\overline t+)<\theta_l$, $u_r(\,\overline t+)>  \theta_r$
and $u(0-,\overline t\,)<A$, then it must be $u(0+,\overline t\,)\geq B$
(see Figure~\ref{figure 3}).
Hence, 
tracing the backward genuine characteristics
(with positive slopes)
through a sequence of points $(x_n,\overline t\,), \, x_n \downarrow 0$,
of continuity for $u$,
one  deduces that
$u_r(\overline t-)=u(0+,\overline t\,)\geq B$,
which is in contrast
with $u_r(\overline t-)<\pi_{r,-}(B)$.
Therefore, $u_l(\,\overline t+)<\theta_l$ and $u_r(\,\overline t+)>  \theta_r$
together imply $u(0-,\overline t\,)\geq A$.
\\
{\bf 2b)} If $u(0-,\overline t\,)>A$, then there should be a shock with negative slope connecting the left state $u(0-,\overline t\,)$ with the right state $u_l(\,\overline t+)=A$ emerging at $(0,\overline t\,)$.
This implies that $u(0-,\overline t\,)<\pi_{l,+}(A)$.
On the other hand, if $u(0-,\overline t\,)\leq \theta_l$ then
tracing the backward genuine characteristics
(with negative slopes)
through a sequence of points $(x_n,\overline t\,), \, x_n \uparrow 0$,
of continuity for $u$,
one deduces that  $u_l(\,\overline t-)=u(0-,\overline t\,)\in (A,\pi_{l,+}(A))$.
This implies that $f_l(u_l(\,\overline t-))<f_l(A)$, which is in contrast with~\eqref{int-entr-cond3}.
Hence,  if $u(0-,\overline t\,)>A$, then it must be $u(0-,\overline t\,)\in (\theta_l, \pi_{l,+}(A))$. 
However, by~\eqref{int-entr-cond3}
we have 
$u_l(\,\overline t-)\in (-\infty, A]\cup [\pi_{l,+}(A),+\infty)$,
which implies 
$u_l(\,\overline t-)\notin
(\theta_l, \pi_{l,+}(A))$. Thus, 
there should be
a shock with positive slope 
arriving in $(0,\overline t\,)$ (or generated in $(0,\overline t\,)$) and
connecting the left state $u(0-,\overline t\,)\in (\theta_l, \pi_{l,+}(A))$ 
with the right state $u_l(\,\overline t-)\in (-\infty, A]\cup [\pi_{l,+}(A),+\infty)$,
which is not entropy admissible for the conservation law with flux $f_l$.
Therefore, $u_l(\,\overline t+)<\theta_l$ and $u_r(\,\overline t+)>  \theta_r$
together imply $u(0-,\overline t\,)= A$,
and with the same arguments we deduce also that $u(0+,\overline t\,)= B$.
\\
{\bf 2c)} If $u(0-,\overline t\,)=A<\theta_l$ and $u(0+,\overline t\,)= B>\theta_r$, then 
tracing the backward genuine characteristics
through two sequences of points $(x_n,\overline t\,), \, x_n \uparrow 0$,
and $(x_n,\overline t\,), \, x_n \downarrow 0$
(having negative and positive slopes, respectively),
one  deduces that there exists $t'<\overline t$ such that
$u_l(t\pm)=A$ and $u_r(t\pm)=B$ for all $t\in (t', \overline t\,)$.
Then, set $\tau\doteq \inf\big\{t'<\overline t;\ u_l(s\pm)=A<\theta_l,\ u_r(s\pm)=B>\theta_r\ \ \forall~s\in(t',\overline t\,)\big\}$. If $\tau>0$, since one has $u_l(\tau+)=A,$ $u_r(\tau+)=B$, repeating the above arguments
of cases {\bf 2a)-2b)} one would deduce that $u_l(t\pm)=A$, $u_r(t\pm)=B$ for all $t\in (t'', \tau)$,
for some $t''<\tau$, which is in contrast with the definition of $\tau$.
Therefore it must be $\tau=0$. On the other hand,  $u_l(\,\overline t+)=A$, $u_r(\,\overline t+)=B$
clearly imply that  $u_l(t\pm)=A$, $u_r(t\pm)=B$ for all $t\in (\,\overline t, t^*)$,
for some $t^*>\overline t$. Thus, one has  $u_l(t\pm)=A$, $u_r(t\pm)=B$ for all $t\in (0, t^*)$
and at any point $(0,t)$, $t\in (0, t^*)$ starts exactly two forward, genuine characteristics $\eta', \eta''$,
which lie in $(-\infty,0)\times (t,  t^*)$ and $(0,+\infty)\times (t,  t^*)$, respectively,
proving property~(i). 
\smallskip

\noindent
{\bf Case 3.}
{\it  $u_l(\,\overline t+)=\theta_l$ \ or \ $u_r(\,\overline t+)=  \theta_r$.}\\
Notice that, by~\eqref{int-entr-cond3}
$u_l(\,\overline t+)=\theta_l$
implies $\theta_l =A$, while
$u_r(\,\overline t+)=  \theta_r$
implies $\theta_r=B$.\\
{\bf 3a)} If $u_l(\,\overline t+)=\theta_l$ and $u(0-,\overline t\,)=\theta_l$, then 
tracing the backward genuine characteristics
through a sequence of points $(x_n,\overline t\,), \, x_n \uparrow 0$,
of continuity for $u$,
one  deduces that $u_l(t)=\theta_l$ for all $t\in (0,\overline t\,)$.
Hence, $u_l(\,\overline t-)=\theta_l$ as well.
In turn, by~\eqref{int-entr-cond3}
this implies that $u_l(\,\overline t\pm)=A$,
$u_r(\,\overline t\pm)\in\{B, \pi_{r,-}(B)\}$.
Suppose that $u_r(\,\overline t+)=B$ and  $u(0+,\overline t\,)<\theta_r$. 
Since by Definition~\ref{int-conn-def} we have $B\geq \theta_r$, it follows that a
shock with positive slope emerges from $(0,\overline t\,)$, and thus 
$u(0+,\overline t\,)>\pi_{r,-}(B)$. 
However, $u_r(\,\overline t-)\in\{B,\pi_{r,-}(B)\}$ 
and $\pi_{r,-}(B)<u(0+,\overline t\,)<B$
imply that there should be
a shock with negative slope 
arriving in $(0,\overline t\,)$ (or generated in $(0,\overline t\,)$) and
connecting the left state $u_r(\,\overline t-)\in\{B, \pi_{r,-}(B)\}$ 
with the right state $u(0+,\overline t\,)\in (\pi_{r,-}(B), B)$, 
which is not entropy admissible for the conservation law with flux $f_r$.
Therefore, if $u_r(\,\overline t+)=B$, then it must be  $u(0+,\overline t\,)\geq \theta_r$. 
Then, tracing the backward genuine characteristics
through a sequence of points $(x_n,\overline t\,), \, x_n \downarrow 0$,
of continuity for $u$,
one  deduces that $u_r(\,\overline t-)\geq\theta_r$. Since 
$u_r(\,\overline t-)\in\{B,\pi_{r,-}(B)\}$, this implies that $u_r(\,\overline t-)=B$.
By similar arguments we deduce that, if $u_r(\,\overline t+)=\pi_{r,-}(B)$,
then also $u_r(\,\overline t-)=\pi_{r,-}(B)$.
Therefore, if $u_l(\,\overline t+)=\theta_l$ and $u(0-,\overline t\,)=\theta_l$, 
it follows that $\overline t$ is a point of continuity for $u_l$ and $u_r$,
$u_l(\,\overline t\pm)=A=\theta_l$, 
and $u_r(\,\overline t\pm)=B$ or $u_r(\,\overline t\pm)=\pi_{r,-}(B)$.
This implies that there is no forward genuine characteristic starting from $(0,\overline t\,)$
and lying on $(-\infty,0)\times (0,+\infty)$, while there is a single 
forward genuine characteristic starting from $(0,\overline t\,)$
and lying on $(0, +\infty)\times (0,+\infty)$, 
which proves the  property (ii).\\
{\bf 3b)} Next, assume that $u_l(\,\overline t+)=\theta_l$ and $u(0-,\overline t\,)> \theta_l$.
Then,  
there should be a shock with negative slope connecting the left state $u(0-,\overline t\,)$
with the right state $\theta_l$ emerging at $(0,\overline t\,)$,
which is not possible since any entropy admissible shock with right state $\theta_l$
has positive slope.
Therefore, $u_l(\,\overline t+)=\theta_l$ implies that $u(0-,\overline t\,)\leq \theta_l$.
\\
{\bf 3c)} Assume now  that $u_l(\,\overline t+)=\theta_l>u(0-,\overline t\,)$.
Then, tracing the backward genuine characteristics (with negative slopes)
through a sequence of points of continuity for $u$
as above,  $(x_n,\overline t\,), \, x_n \uparrow 0$,
we deduce that $u_l(\,\overline t-)=u(0-,\overline t\,)$.
Since $\theta_l =A\neq u_l(\,\overline t-)$,
by~\eqref{int-entr-cond3}
this implies that $u_r(\,\overline t-)\leq  \pi_{r,-}^l(A)=\pi_{r,-}(B)$.
On the other hand, by the same observations in case {\bf 1a)}
we know that 
$u_r(\,\overline t+)\in\{B, \pi_{r,-}(B)\}$.
Moreover, with similar arguments of case {\bf 1a)}
we deduce that
$u_r(\,\overline t+)=B$ and $u_r(\,\overline t-)\leq \pi_{r,-}(B)$
imply  $u(0+,\overline t\,)\geq \theta_r$,
and $u_r(\,\overline t-)=B$.
Next, assume that $u_r(\,\overline t-)\leq \pi_{r,-}(B)$, $u_r(\,\overline t+)=\pi_{r,-}(B)$. 
Again with similar arguments as above
we deduce that $u_r(\,\overline t+)=\pi_{r,-}(B)$ implies $u(0+,\overline t\,)=\pi_{r,-}(B)$,
and that there is no entropy admissible shock connecting a left state $u_r(\,\overline t-)<  \pi_{r,-}(B)$
with a right state $\pi_{r,-}(B)$.  Hence, if $u_r(\,\overline t-)\leq  \pi_{r,-}(B)$, $u_r(\,\overline t+)=\pi_{r,-}(B)$,
it must be $u_r(\,\overline t-)= \pi_{r,-}(B)$.
In turn, because of~\eqref{int-entr-cond3}
and since $u_l(\,\overline t-)<\theta_l$, this implies that $u_l(\,\overline t-)=A$,
which is in contrast with $u_l(\,\overline t-)=u(0-,\overline t\,)<\theta_l=A$.\\
Therefore, $u_l(\,\overline t+)=\theta_l$
implies that $u(0-,\overline t\,)=\theta_l$ as well, which are the assumptions of case  {\bf 1a)},
and thus property (ii) is verified.
Moreover, one has $u_r(\,\overline t\pm)=u(0+,\overline t\,)\in\{B,\pi_{r,-}(B)\}$.
With similar arguments we deduce that $u_r(\,\overline t+)=  \theta_r$
implies $u(0+,\overline t\,)=\theta_r$, $u(0-,\overline t\,)=u_l(\,\overline t\pm)
\in\{A,\pi_{l,+}(A)\}$,
and
then the same conclusions of the case $u_l(\,\overline t+)=\theta_l$ hold true.
This completes the proof of the Proposition. 
\end{proof}

\begin{rem}
\label{charact-lrtraces-admissible}
By the analysis of Proposition~\ref{Norarefactions}
it follows that, if $u_l\doteq u_l(\,\overline t\,), u_r\doteq u_r(\,\overline t\,)$, are the
one-sided limits~\eqref{traces} at $x=0$, and $\overline t>0$, of an $AB$-entropy solution,
then either $(u_l,u_r)=(A,B)$, or there exists a backward characteristic through $(0,\overline t\,)$, defined 
on $[0,\overline t\,]$, and taking values in $\mathbb{R}\setminus\{0\}$ at any time $t<\overline t$.
In this latter case,  consider the minimal and maximal backward characteristics $\xi_-, \xi_+$  
through $(0,\overline t\,)$, defined 
on~$[0,\overline t\,]$, and taking values in $\mathbb{R}$.
By the proof of Proposition~\ref{Norarefactions},
and recalling the definition~\eqref{lrtraces-admissible-def},
we deduce that
one of the following cases occurs:
\vspace{-3pt}
\begin{itemize}
\item[1.] $\xi_{\pm}(0)<0$ and $(u_l,u_r)\in\mathcal{T}_1$\,;
\vspace{-5pt}
\item[2.] $\xi_{\pm}(0)>0$ and $(u_l,u_r)\in\mathcal{T}_2$\,;
\vspace{-5pt}
\item[3.] $\xi_-(0)<\xi_+(0)=0$, or $\xi_-(0)=0<\xi_+(0)$, or $\xi_-(0)<0<\xi_+(0)$,
and $(u_l,u_r)\in\mathcal{T}_3$\,.
\end{itemize}
\end{rem}
\medskip

\noindent
The next result shows that the upper bounds on the
Dini derivative of a function $\omega\in \mathcal{A}_i(T)$, $i=1,2,3$,
given in~\eqref{bound totale1}, \eqref{bound totale2}, \eqref{bound totale3}, 
are equivalent to the monotonicity of the maps $\varphi_i$
that associates to any $x\neq 0$,
 the starting point $\varphi_i(x)$ at time $t=0$ of a characteristic that 
 reaches $x$ at time $T$.

\begin{lem}\label{lemma bound}
Let $\omega:\mathbb{R}\rightarrow\mathbb{R}$  be a bounded 
function having right and left limits in any point. Then, the following hold.
\vspace{-5pt}
\begin{itemize}
\item[(i)]
If  $\big(\omega(0-),\, \omega(0+)\big)\in \mathcal{T}_1$\,,
and $\omega$ satisfies~\eqref{c1}, then~\eqref{bound totale1} holds
if and only if the function
\vspace{-5pt}
\begin{equation}
\varphi_1 (x):=\begin{cases}
x -f'_{l}(\omega(x))\cdot T & \text{ if }  x<0,\\
-f'_{l}\circ f_{l,+}^{-1}\circ f_{r}(\omega(x))\cdot\big(T-{x}/{f'_{r}(\omega(x))}\big) & \text{ if } 0<x<R,\\
x -f'_{r}(\omega(x))\cdot T & \text{ if } x>R
 \end{cases} 
\label{phi1-def}
\end{equation}
is nondecreasing, and the
function 
\begin{equation}
\label{psi1-def}
\psi_1(x):=T-{x}/{f'_r(\omega(x))}
\qquad 0<x<R\,,
\end{equation}
\vspace{-20pt}

is decreasing.
\item[(ii)]
If  $\big(\omega(0-),\, \omega(0+)\big)\in \mathcal{T}_2$\,, and $\omega$ satisfies~\eqref{c4}, then 
~\eqref{bound totale2} holds if and only if the function
\vspace{-10pt}

\begin{equation}
 \varphi_2 (x):= 
 \begin{cases}
x-f'_{l}(\omega(x))\cdot T & \text{ if } x<L,\\
-f'_{r}\circ f_{r,-}^{-1}\circ f_{l}(\omega(x))\cdot\big(\text{T}-{x}/{f'_{l}(\omega(x))}\big) & \text{ if } L<x<0,\\
x -f'_{r}(\omega(x))\cdot T & \text{ if } x>0
\end{cases}
\label{phi2-def}
\end{equation}
 is nondecreasing, and the
function 
\vspace{-10pt}

\begin{equation}
\label{psi2-def}
\psi_2(x):=T-{x}/{f'_l(\omega(x))}
\qquad L<x<0\,,
\end{equation}

\vspace{-10pt}
is increasing.
 \item[(iii)]
 If  $\omega$ satisfies $(\ref{T3-cond1})-(\ref{c7})$,  then the function
 \vspace{-10pt}
 
\begin{equation} 
\varphi_3 (x):= \begin{cases}
x-f'_{l}(\omega(x))\cdot T & \text{ if } x<L,\\
\noalign{\smallskip}
x -f'_{r}(\omega(x))\cdot T & \text{ if } x>R,
\end{cases}
\label{phi3-def}
\end{equation}
 is nondecreasing 
 if and only if~\eqref{bound totale3} holds. 
\end{itemize}
\end{lem}


\begin{proof}
We prove only the  statement (i), 
the proofs of the other two statements being entirely similar. \\
\noindent
{\bf 1.} First observe that 
the  monotonicity of $\varphi_1$,  $\psi_1$, are equivalent to 
\begin{equation}\label{monotonia1}
D^+\varphi_1(x)\geq 0\quad\ \forall~x\in \mathbb{R},
\qquad\quad  D^+\psi_1(x)<0
\quad\ \forall~x\in (0,R)\,.
\end{equation}
Next, notice that by~\eqref{c1} we have
\begin{align}
&f'_l(\omega(0-))\geq 0\,,\qquad\qquad\qquad 
\omega(0-)\geq \theta_l\,,
\qquad \qquad \ f'_r(\omega(x))\cdot T-x\geq 0\quad \forall~x\in (0,R)\,,
\label{cond1-monot}
\\
\noalign{\smallskip}
&R-f'_r(\omega(R+))\cdot T\geq 0\,,\qquad\ 
f'_r(\omega(R-))>0\,,
\qquad\ T-{R}/{f'_r(\omega(R-))}\geq 0\,.
\label{cond2-monot}
\end{align}
Moreover,
$(\omega(0-),\, \omega(0+))\in \mathcal{T}_1$
implies that $f_l(\omega(0-))\geq f_r(\omega(0+))$.
Hence, relying on~\eqref{cond1-monot} we deduce that
\begin{equation}
\omega(0-)= f_{l,+}^{-1}\circ f_l (\omega(0-)) \geq 
f_{l,+}^{-1}\circ f_r (\omega(0+)),
\end{equation}
which in turn, together with~\eqref{cond2-monot},
yields
\begin{equation}
\varphi_1(\omega(0-))=-f'_l(\omega(0-))\cdot T
\leq - f'_{l}\circ f_{l,+}^{-1}\circ f_r(\omega(0+))\cdot T = \varphi_1(\omega(0+)).
\end{equation}
On the other hand, since the function $f_{l,+}^{-1}$ takes values in $[\theta_l,+\infty)$
(see definition in Section~\ref{sec:main results}), it follows that 
\begin{equation}
\label{cond3-monot}
f'_{l}\circ f_{l,+}^{-1}\circ f_r (v)\geq 0\qquad\quad \forall~v\in\mathbb{R}\,.
\end{equation}
Hence, because of~\eqref{cond2-monot}, we deduce that
\begin{equation}
\varphi_1(R-)=
-f'_{l}\circ f_{l,+}^{-1}\circ f_{r}(\omega(R-))\cdot\big(\text{T}-{R}/{f'_{r}(\omega(R-))}\big)
\leq 0\leq R-f'_r(\omega(R+))\cdot T
= \varphi_1(R+).
\end{equation}
Therefore, in order to establish the statement (i) it is sufficient to show that
\begin{equation}\label{monotonia2}
D^+\varphi_1(x)\geq 0 \qquad\ \forall~x\in \mathbb{R}\setminus\{0,R\}
\qquad\quad  D^+\psi_1(x)<0
\quad\ \forall~x\in (0,R),
\end{equation}
are verified if and only if~\eqref{bound totale1} holds.
\\
{\bf 2.}
We first show that the equivalence between~\eqref{bound totale1} and~\eqref{monotonia2}
holds at any point of discontinuity for $\omega$.
To this end observe that
the maps
\vspace{-3pt}
\begin{equation*}
g_1(v,x)\doteq  x-f'_{l}(v)\cdot T,\quad\
g_2(v.x)\doteq 
\left[-f'_l\circ f_{l,+}^{-1}\circ f_r(v)\big(\text{T}-{x}/{f'_r(v)}\big)\right]_{\mid \{v;\, f'_r(v)\cdot T\!-\!x\geq 0\}},
\quad\ g_3(v,x)\doteq x -f'_{r}(v)\cdot T\,,
\end{equation*}
\vspace{-3pt}
are nonincreasing in $v$ since, by the strict convexity of the fluxes $f_l, f_r$,
and because of~\eqref{cond3-monot},
we have
\begin{equation}
\begin{aligned}
\partial_v\, g_1(v,x)&=-f''_l(v)\cdot T <0\,,
\\
\partial_v\, g_2(v,x)&=-
\dfrac{\big[f''_l\circ f_{l,+}^{-1}\circ f_r(v)\big]\big[f'_r(v)\big]^2\big[f'_r(v)\cdot T-x
\big]+x \big[f'_l\circ f_{l,+}^{-1}\circ f_r(v)\big]^2 \big[f''_r(v)\big]}{\big[f'_l\circ f_{l,+}^{-1}\circ f_r(v)\big]\big[f'_r(v)\big]^2}\leq 0\,,
\\
\partial_v\, g_3(v,x)&=-f''_r(v)\cdot T <0\,.
\end{aligned}
\end{equation}
Moreover, \eqref{bound totale1}, \eqref{cond1-monot} 
and the assumption {\bf H1)} 
together imply that $D^+ \omega(x)$ is upper bounded
since
\vspace{-5pt}
\begin{equation}
\label{cond4-monot}
D^+ \omega(x) \leq
\begin{cases}
&\hspace{-0.12in}{1}/{(c \cdot T)}
\qquad\qquad\ \text{if}\quad x<0\,,\\
\noalign{\smallskip}
&\hspace{-0.12in} {f'_r(\omega(x))}/{(x\cdot  c)}
\quad\,  \text{if}\quad 0<x<R\,,\\
\noalign{\smallskip}
&\hspace{-0.12in}{1}/{(c \cdot T)} 
 \qquad\qquad\ \text{if}\quad x>0\,.
 \end{cases} 
\end{equation}
Hence, if $x$ is a point of discontinuity for $\omega$, the inequality~\eqref{bound totale1} 
is verified if and only if $\omega(x-)>\omega(x+)$. On the other hand, since
\vspace{-15pt}

\begin{equation*}
\varphi_1(x)=
\begin{cases}
g_1(\omega(x),x)&\ \ \text{if}\quad x<0\,,
\\
g_2(\omega(x),x)&\ \ \text{if}\quad 0<x<R\,,
\\
g_3(\omega(x),x)&\ \ \text{if}\quad x>R\,,
\end{cases}
\end{equation*}
by the monotonicity of the maps $g_1, g_2, g_3$
in $v$, and by the strict convexity of $f_r$,
we have $\omega(x-)>\omega(x+)$ if and only if $\varphi_1(x-)<\varphi_1(x+)$
and $\psi_1(x-)>\psi_1(x+)$ (if $x\in (0,R)$). In turn, if $x$ is a point of discontinuity for~$\varphi_1$
and $\psi_1$ (if $x\in (0,R)$),
then $\varphi_1(x-)<\varphi_1(x+)$, $\psi_1(x-)>\psi_1(x+)$,
are verified if and only if $D^+\varphi_1(x)\geq 0$, 
$D^+\psi_1(x)<0$. Thus, we conclude that in order to establish
the statement (i) it is sufficient to prove that the equivalence between~\eqref{bound totale1} and~\eqref{monotonia2}
is verified at any point of continuity for $\omega$.\\
{\bf 3.}
If $x<0$ is a point of continuity for $\omega$, then we get
\vspace{-5pt}

\begin{equation}
\label{limsup1}
\begin{aligned}
    D^+\varphi_1(x)&
    =  1 + \partial_v\,g_1(\omega(x),x) \cdot D^+\omega(x)
    = 1-f''_{l}(\omega(x)) \cdot T\cdot D^+\omega (x)\,,
\end{aligned}
\end{equation}
\vspace{-10pt}

\noindent
which shows the equivalence between the first inequality in~\eqref{bound totale1} and~\eqref{monotonia2}.
With the same computation we find the equivalence between the third inequality in~\eqref{bound totale1} and~\eqref{monotonia2}, considering a point $x>R$ of continuity for $\omega$.
Next, consider a point $ 0< x< R$ where $\omega$ is continuous. Then we find
\vspace{-3pt}
\begin{equation*}
\label{gran catena}
\begin{aligned}
D^+\varphi_1(x)&= 
\partial_v\,g_2(\omega(x)) \cdot D^+\omega(x) +  \frac{\big[f'_l\circ f_{l,+}^{-1}\circ f_r(\omega(x))\big]}{ f'_r(\omega(x))}
\\ 
\noalign{\smallskip}
&=
-\frac{ \big[f''_l\circ f_{l,+}^{-1}\circ f_r(\omega(x))\big]\!\left[f'_r(\omega(x))\right]^2\big( f'_r(\omega(x))\cdot T\!-\!x \big)\!+\!x\big[f'_l\circ f_{l,+}^{-1}\circ f_r(\omega(x))\big]^2  f''_r(\omega(x))}{\big[f'_l\circ f_{l,+}^{-1}\circ f_r(\omega(x))\big]  \cdot\left[f'_r(\omega(x))\right]^2}\cdot D^+\omega(x)+
\\
&\quad +\frac{\big[f'_l\circ f_{l,+}^{-1}\circ f_r(\omega(x))\big]}{ f'_r(\omega(x))},
\end{aligned}
\end{equation*}
and
\vspace{-5pt}
\begin{equation*}
D^+\psi_1(x)=-\frac{f'_r(\omega(x))-xf''_r(\omega(x))\cdot D^+\omega(x)}{\left[f'_r(\omega(x))\right]^2}\,.
\end{equation*}
Hence,  by~\eqref{cond3-monot} we deduce that $D^+\varphi_1(x)\geq 0$
and  $D^+\psi_1(x)< 0$ hold if and only if
\begin{equation}
\label{limsup111}
\begin{aligned}
&\left[
\big[f''_l\circ f_{l,+}^{-1}\circ f_r(\omega(x))\big]\!\left[f'_r(\omega(x))\right]^2\big( f'_r(\omega(x))\cdot T\!-\!x \big)\!+\!x\big[f'_l\circ f_{l,+}^{-1}\circ f_r(\omega(x))\big]^2  f''_r(\omega(x))
\right]\cdot D^+\omega(x)\leq 
\\
&\qquad \leq \big[f'_l\circ f_{l,+}^{-1}\circ f_r(\omega(x))\big]^2\cdot f'_r(\omega(x))\,,
\end{aligned}
\end{equation}
and
\begin{equation}
\label{limsup112}
x\, f''_r(\omega(x)) \cdot D^+\omega(x)< f'_r(\omega(x))\,.
\end{equation}
By~\eqref{cond1-monot} and the convexity of $f_r$, 
the inequalities~\eqref{limsup111}-\eqref{limsup112} are equivalent to the
second inequality in~\eqref{bound totale1}, and the 
prove of the statement (i) is completed.
\end{proof}
\medskip
An immediate consequence of Lemma~\ref{lemma bound} is the following.
%
\begin{lem}\label{lem2}
In the same setting and with the same notations of Theorem~\ref{Teo1}, 
the sets $\mathcal{A}_1(T), \mathcal{A}_2(T), \mathcal{A}^{AB}_3(T)$ are equivalently defined 
as sets of 
functions $\omega\in\mathbf{L}^\infty(\mathbb{R})$ having  essential left and right limits at $x=0$,
that satisfy the following conditions.  
\vspace{-5pt}
 \begin{itemize}
  \item[$\mathcal{A}_{1}\text{(T)}$] 
  is the set  of   all functions $\omega$  that satisfy\, $\big(\omega(0-),\, \omega(0+)\big)\in \mathcal{T}_1$\,,
  and for which there exists  $R>0$ such that: there holds $\omega(R-)\geq \omega(R+)$,
\vspace{-3pt}
\begin{equation}
f'_l(\omega(x)) \geq {x}/{T} +f'_l(\omega(0-)) \  \ \ \forall\,x<0,
\qquad
f'_r(\omega(x))\geq {x}/{T}  \ \ \ \forall\,0<x<R,
\qquad
f'_r(\omega(x)) < {x}/{T} \ \ \ \forall\,x>R,
\label{c11} 
\end{equation}
\vspace{-20pt}

\noindent
 the map $\varphi_1$ in~\eqref{phi1-def} is nondecreasing, and the map $\psi_1$ in~\eqref{psi1-def} is decreasing. 
\vspace{-2pt}
  \item[$\mathcal{A}_{2}\text{(T)}$] is the set  of   all functions $\omega$  that satisfy\, $\big(\omega(0-),\, \omega(0+)\big)\in \mathcal{T}_2$\,,
  and for which there exists  $L<0$ such that: there holds $\omega(L-)\geq \omega(L+)$,
\vspace{-3pt}
\begin{equation}
f'_l(\omega(x)) >{x}/{T}  \ \  \ \forall\,x<L,
\qquad
f'_l(\omega(x))\leq {x}/{T}  \ \ \  \forall\,0<x<L,
\qquad\
f'_r(\omega(x)) \leq  {x}/{T}+f'_r(\omega(0+)) \ \ \ \forall\,x>0,
\label{c21} 
\end{equation}
\vspace{-20pt}

\noindent
 the map $\varphi_2$ in~\eqref{phi2-def} is nondecreasing, and the map $\psi_2$ in ~\eqref{psi2-def} is increasing.
  \vspace{-2pt}
  \item[$\mathcal{A}_3^{AB}\text{(T)}$] is the set  of all functions $\omega$
    for which there exist  $L\leq 0 \leq R$, such that: 
    \vspace{-5pt}
  \begin{equation}
  \label{T3-cond11}
  \big(\omega(0-),\, \omega(0+)\big)\in
  \begin{cases}
   \mathcal{T}_{3,-}\cup \mathcal{T}_{3,+}\  &\text{if} \quad L=R=0\,,
   \\
   \noalign{\smallskip}
   \big\{(A,B)\big\}\  &\text{if} \quad L\leq 0 \leq R\,,
   \end{cases}
   \qquad\quad
   \omega(L-)\geq \omega(L+)\,,\qquad\quad 
   \omega(R-)\geq \omega(R+)\,,
  \end{equation}
   \begin{equation}
    \label{c71}
    \begin{aligned}
    &\omega(x)=A\quad \forall~x\in (L,0),\qquad\quad  \omega(x)=B\quad\forall~x\in (0,R),
    \\
    \noalign{\medskip}
     &f'_l(\omega(x)) \geq 
      \begin{cases}
   {x}/{T}\ \ &\text{if}\ \ L<0,
     \\
     \noalign{\smallskip}
     {x}/{T}\!+\!f'_l(\omega(0-)) &\text{if} \ \ L=0,
     \end{cases}
    \qquad \forall~x\in (-\infty, L),
     \\
    \noalign{\medskip}
     &
     f'_r(\omega(x)) \leq 
      \begin{cases}
     {x}/{T}\ \ &\text{if}\ \ R<0,
     \\
     \noalign{\smallskip}
     {x}/{T}\!+\!f'_r(\omega(0+)) &\text{if} \ \ R=0,
     \end{cases}
     \qquad \forall~x\in (R,+\infty),
     \end{aligned}
     \end{equation}
     and the map $\varphi_3$ in~\eqref{phi3-def} is nondecreasing.
\end{itemize}
\end{lem}


\section{Proof of Theorem \ref{Teo1}}
We proceed by dividing the proof into two steps: first we show that any 
attainable profile at time $T>0$
of a solution to the problem~\eqref{eq1},\eqref{flux}-\eqref{datum} satisfies all the conditions of one of the tree sets described in the statement of Lemma~\ref{lem2}.
Next, we prove that, for any function $\omega$ in $\mathcal{A}_{1}(T)$, $\mathcal{A}_{2}(T)$ and $\mathcal{A}_{3}^{AB}(T)$, there exists $\overline{u}\in {\bf L}^{\infty}(\mathbb{R})$ such that $ S_{T} \overline{u}=\omega $.

\subsection{Proof of \ ${\bf \mathcal{A}(T)\subseteq \mathcal{A}_{1}(T)\cup\mathcal{A}_{2}(T)\cup\mathcal{A}_{3}^{AB}(T)}$.} 

Given $\overline {u}\in {\bf L^{\infty}}$, let $u(\cdot, t)\doteq \mathcal{S}_t^{AB}\overline u$, $t>0$,
we will show that $\omega\doteq \mathcal{S}_T^{AB}\overline u$ belongs to one
of the sets $\mathcal{A}_{1}(T), \mathcal{A}_{2}(T), \mathcal{A}_{3}^{AB}(T)$.
By Remark~\ref{propert1-AB-sol} we know that $\omega \in BV_{loc}(\mathbb{R}\setminus\{0\})$
and that $\omega$ admits one-sided limits at $x=0$. 
Then, recalling Remark~\ref{charact-lrtraces-admissible}, we will distinguish the following five cases.

\noindent
{\bf Case 1.} $\omega(0-)=A<\theta_l$, \
$\omega(0+)=B>\theta_r$.\\
Observe that, tracing the backward characteristics through points of continuity of $\omega$
in a neighbourhood of $x=0$,
with the same arguments of the proof of 
Proposition~\ref{Norarefactions} and relying on~\eqref{int-entr-cond3},
we deduce that 
\vspace{-5pt}
\begin{equation}
\label{411}
\begin{aligned}
&\qquad\qquad\qquad\qquad
u_l(t)=A,\qquad\quad u_r(t)=B\qquad \forall~t\in(\delta_1, T)\,,
\\
&\omega(x) =A\qquad \forall~x\in(-\delta_1,0),
\qquad\quad
\omega(x) =B\qquad \forall~x\in(0,\delta_1),
\end{aligned}
\end{equation}
\vspace{-8pt}

\noindent
for some there exist $\delta_1>0$ such that  Thus, by Proposition~\ref{Norarefactions}  we deduce that
\vspace{-5pt}
\begin{equation}
\label{412}
u_l(t)=A,\qquad u_r(t)=B\qquad\quad\ \forall~t\in(0, T)\,.
\end{equation}
\vspace{-15pt}

\noindent
Next, let $R\doteq \sup\{x>0;\, \omega(x)=B\ \text{for all}~y\in(0,x)\}$,
 $L\doteq \inf\{x<0;\, \omega(x)=A\ \text{for all}~y\in(x,0)\}$.
 By~\eqref{411} one has $L<0<R$. 
 Notice that $\omega(L-)\geq \omega(L+)$ and 
 $\omega(R-)\geq \omega(R+)$ because of the Lax entropy condition
 (see Remark~\ref{lax-kruzhkov-entr}).
 Consider the maximal backward characteristic $\xi_{R,+}$
 through $(R,T)$ and assume that it crosses  the axis $x=0$
 at time $t_R>0$. Then, by~\eqref{412} and the observations in Section~\ref{sec-struct-sol},
 it follows that  $\xi_{R,+}$ is a segment with positive slope $f'_r(B)
 =f'(\omega(R+))$.
 But this means that we may find $\delta_2>0$ such that
 all backward characteristics $\xi_x$ through points $(x,T)$, with $x\in (R, R+\delta_2)$,
 reach the axis $x=0$ at times $t_x\in (\delta_2, t_R)$. This implies that $\omega(x)=u_r(t_x)=B$
 for all $x\in (R, R+\delta_2)$, which is in contrast with the definition of $R$. Thus,
  the maximal backward characteristic $\xi_{R,+}$ is defined on the whole interval $[0,T]$,
  and there holds $\xi_{R,+}(t)\geq 0$ for all $t\in [0,T]$.
 With the same arguments we deduce that the
 minimal backward characteristic $\xi_{L,-}$
 through $(L,T)$, is defined on $[0,T]$
  and there holds $\xi_{L,-}(t)\leq 0$ for all $t\in [0,T]$.

 Given any $x>R$, consider the  minimal and maximal backward characteristics $\xi_{x,-}, \xi_{x,+}$
 through $(x,T)$. Since $\xi_{x,\pm}, \xi_{R,+}$
 are genuine characteristics for the conservation law $u_t+f_r(u)_x=0$,
 it follows that they never intersect in the open quarter of plane
 $(0,+\infty)\times(0,+\infty)$. Hence, $\xi_{x,\pm}$ are defined on the whole 
 interval $[0,T]$,
 and there holds
 \vspace{-5pt}
 \begin{equation*}
 \xi_{x,-}(t)= x+f'_r(\omega(x-))\cdot (t-T),\qquad\quad
 \xi_{x,+}(t)= x+f'_r(\omega(x+))\cdot (t-T)\qquad\forall~t\in[0,T]\,.
  \end{equation*}
  \vspace{-15pt}
  
\noindent
Moreover, one has $x-f'_r(\omega(x\pm))\cdot T=\xi_{x,\pm}(0)\geq \xi_{R,+}(0)\geq 0$, which implies $f'_r(\omega(x\pm))\leq \frac{x}{T}$.
%
 On the other hand,
 recalling the definition~\eqref{phi3-def} of $\varphi_3$, we deduce that, for every $R<x<y$, there holds
$\varphi_3(x\pm)=\xi_{x,\pm}(0)\leq \xi_{y,\pm}(0)=\varphi_3(y\pm)$,
which proves the nondecreasing monotonicity of $\varphi_3$ on $(R, +\infty)$.
With similar arguments we deduce that $f'_l(\omega(x\pm))\geq \frac{x}{T}$
for all $x\in (-\infty, L)$,
and that $\varphi_3$ is nondecreasing also on $(-\infty,L)$.
Therefore, the function $\omega$ satisfies conditions~\eqref{T3-cond11}, \eqref{c71}
and $\varphi_3$ is nondecreasing on $(-\infty,L)$ and $(R,+\infty)$.
Since $\varphi_3(x)\leq 0$ for all $x\in (-\infty, L)$, and $\varphi_3(x)\geq 0$ for all $x\in (R,+\infty)$,
it follows that $\varphi_3$ is nonincresing on its domain
and hence we have $\omega\in \mathcal{A}_{3}^{AB}(T)$.
\smallskip

\noindent
{\bf Case 2.} $(\omega(0-), \omega(0+))=(A,B)$, \ $A=\theta_l, B>\theta_r$,\
or \ $A<\theta_l, B=\theta_r$,\
or \ $A=\theta_l, B=\theta_r$.\\
Assume that $A=\theta_l, B>\theta_r$, the other cases being entirely similar.
Then, letting $R\doteq \sup\{x>0;\, \omega(x)=B\ \text{for all}~y\in(0,x)\}$,
by the same analysis of Case 1 we deduce that $R>0$, $\omega(R-)\geq \omega(R+)$,
$f'_r(\omega(x\pm))\leq \frac{x}{T}$ for all $x>R$, and that the map $\varphi_3$
in~\eqref{phi3-def} is nondecreasing on $(R, +\infty)$. Next, assume that there exists $x<0$ such that 
$f'_l(\omega(x+))< \frac{x}{T}$. Then, the maximal backward characteristics $\xi_x$ starting
at $(x,T)$ crosses the axis $x=0$ at some time $t_x>0$. On the other hand, 
 the maximal backward characteristics $\xi_{x_n}$ trough
a sequence of points $(x_n,T), x_n\uparrow 0$, are lines with slope
$f'_l(\omega(x_n+))\to f'_l(\omega(0-))=0$. Hence, there will be some $n$
such that $\xi_{x_n}$ intersect $\xi_x$ in $(-\infty,0)\times (0,+\infty)$, which is
not possible. Therefore, there holds $f'_l(\omega(x\pm))\geq  \frac{x}{T}$ for all $x<0$,
and with the same arguments of Case 1 one can show that $\varphi_3$
is nondecreasing on $(-\infty, 0)$ as well,
and that $\varphi_3(0-)\leq 0<\varphi_3(R+)$. Thus, setting $L=0$, 
we have shown that $\omega\in \mathcal{A}_{3}^{AB}(T)$.
\smallskip

\noindent
{\bf Case 3.} $(\omega(0-), \omega(0+))\in\mathcal{T}_1$.\\
Notice that $(\omega(0-), \omega(0+))\in\mathcal{T}_1$ implies $\omega(0+)>\theta_r$,
and hence $f'_r(\omega(0+))>0$. Thus, there exist $\delta_1>0$ such that
$f'_r(\omega(x+))\geq \frac{x}{T}$ for all $x\in (0,\delta_1)$. Then, setting
$R\doteq \sup\{x>0;\, f'_r(\omega(x+))\geq \frac{x}{T}\}$, one has $R>0$
and $\omega(R-)\geq \omega(R+)$, because of the Lax entropy condition
 (see Remark~\ref{lax-kruzhkov-entr}).
Observe that if, $f'_r(\omega(x+))< \frac{x}{T}$ or $f'_r(\omega(x-))< \frac{x}{T}$
for some $x\in (0,R)$, then one would deduce that the backward 
(minimal and maximal) characteristics 
$\xi_{y,\pm}$ through $(y,T)$, $y\in (x,R)$, should cross in $(0,+\infty)\times(0,+\infty)$
the backward characteristic 
$\xi_{x,+}$ or $\xi_{x,+}$ through $(x,T)$, which is not possible.
Thus, there holds $f'_r(\omega(x\pm))\geq \frac{x}{T}$ for all $x\in(0,R)$.
Next, consider the maximal backward characteristic $\xi_{R,+}$ through $(R,T)$,
and suppose that it is defined on an interval $[t_R, T]$, $t_R>0$, with $\xi_{R,+}(t_R)=0$.
This means that $f'_r(\omega(R+))=\frac{R}{t_R}>\frac{R}{T}$, which implies 
that there exists $\delta_1>R$ such that
$f'_r(\omega(x+))> \frac{x}{T}$ for all $x\in (R,\delta_1)$. 
But this is in contrast with the definition of $R$. Hence $\xi_{R,+}$
is defined on the whole interval $[0,T]$,
  and there holds $\xi_{R,+}(t)\geq 0$ for all $t\in [0,T]$.
On the other hand, $(\omega(0-), \omega(0+))\in\mathcal{T}_1$
implies $\omega(0-)>\theta_l$, and hence
the minimal backward characteristics $\xi_{0,-}$
through~$(0,T)$ satisfies $\xi_{0,-}(0)=-f'_l(\omega(0-))\cdot T<0$.
Then, since backward characteristics starting at points $(x,T)$
with $x<0$ or $x>R$ cannot cross $\xi_{0,-}$ and $\xi_{R,+}$, respectively,
and by the definition of $R$,
we deduce that $f'_l(\omega(x\pm))\geq  \frac{x}{T}+f'_l(\omega(0-))$
for all $x\in (-\infty, 0)$ and $f'_r(\omega(x\pm))< \frac{x}{T}$
for all $x\in (R, +\infty,)$.
Moreover, with the same arguments we deduce that 
$f'_r(\omega(x\pm))\geq  \frac{x}{T}$
for all $x\in (0,R)$. Therefore, the function $\omega$ satisfies 
condition~\eqref{c11}.

Next, with similar arguments of Case 1, we deduce that the map 
$\varphi_1$ defined in~\eqref{phi1-def} is nondecreasing on the intervals
$(-\infty, 0)$ and $(R, +\infty)$.
Regarding the monotonicity of $\varphi_1, \psi_1$ (defined in~\eqref{psi1-def}) on $(0,R)$, 
first observe that, since the Lax entropy condition implies $\omega(x-)\geq \omega(x+)$,
by the strict convexity of $f_r$
it follows that $\psi_1(x-)> \psi_1(x+)$ at any point  $x\in (0,R)$
of discontinuity for $\omega$.
Next, consider the
maximal backward characteristic $\xi_{x,+}$ through $(x,T)$, $0<x<R$,
 and the minimal backward characteristic~$\xi_{y,-}$ through $(y,T)$, $x<y<R$, given by
\begin{equation*}
 \xi_{x,+}(t)= x+f'_r(\omega(x+))\cdot (t-T)
 \quad t \in [t_x, T],\qquad\quad
 \xi_{y,-}(t)= y+f'_r(\omega(y-))\cdot (t-T)
 \quad t\in [t_y, T]\,,
  \end{equation*}
  with $\xi_{x,+}(t_x)= \xi_{y,-}(t_y)=0$, $t_x\doteq \psi_1(x+)$, $t_y\doteq \psi_1(y-)$.
Since $\xi_{x,+}$,  $\xi_{y,-}$ cannot cross on
$(0,+\infty)\times(0,+\infty)$, 
one has $t_x\geq t_y$.
On the other hand, if $t_x=t_y$, then
there would be two forward characteristics with positive slope issuing form $(0,t_x)$,
which is in contrast with Proposition~\ref{Norarefactions}. Thus, it must be $\psi_1(x+)=t_x>t_y= \psi_1(y-)$,
which proves the decreasing monotonicity of $\psi_1$.

The monotonicity of $\psi_1$ in particular implies $\psi_1(x\pm)>\psi_1(R-)$ for all $x\in (0,R)$.
Observe that $u_r(t\pm)>\theta_r$\linebreak for all $t\in (\psi_1(R-),T)$,
since any point $(0,t), t\in (\psi_1(R-),T)$ is reached by a backward characteristic 
(crossing $x=0$ with positive slope) issuing from a point 
$(x,T), x \in (0,R)$. In turn, this implies that $u_l(t\pm)>\theta_l$ 
for any time $t\in (\psi_1(R-),T)$
of continuity for $u_l, u_r$,
since otherwise, by~\eqref{int-entr-cond3}
we should have $u_l(\,\overline t-)=A$, $u_r(\,\overline t-)=B$,
for some $\overline t \in  (\psi_1(R-),T)$. But, by the analysis of Proposition~\ref{Norarefactions}, this 
implies that either
\vspace{-3pt}
\begin{equation*}
u_l(t)=A,\qquad u_r(t)=B \quad\ \forall~t\in (0,\overline t\,),
\qquad\quad u_l(\,\overline t+)>\theta_l\,,\qquad
u_r(\,\overline t+)<\theta_r\,,
\end{equation*}
\vspace{-5pt}
or
\begin{equation*}
u_l(t)=A,\qquad u_r(t)=B \quad\ \forall~t\in (0,T),
\end{equation*}
which are in contrast with $u_r(t\pm)>\theta_r$ for all $t\in (\psi_1(R-),T)$,
and with $(\omega(0-), \omega(0+))\in\mathcal{T}_1$, respectively.
Therefore, we have $u_l(t\pm)>\theta_l$ 
for all $t\in (\psi_1(R-),T)$. Hence, by~\eqref{int-entr-cond3}, \eqref{pimap-def},
there holds 
$u_l(t)=\pi_{l,+}^r(u_r(t))$
at any time $t\in (\psi_1(R-),T)$
of continuity for $u_l, u_r$. 
Hence, in particular for $t_x\doteq \psi_1(x+)$, $t_y\doteq \psi_1(y-)$  we find
\vspace{-1pt}
\begin{equation}
u_l(t_x-)=\pi_{l,+}^r(u_r(t_x-))=\pi_{l,+}^r(\omega(x+))\,,
\qquad\quad
u_l(t_y+)=\pi_{l,+}^r(u_r(t_y+))=\pi_{l,+}^r(\omega(y-))\,.
\end{equation}
\vspace{-10pt}

\noindent
Consider now the backward characteristics (for $u_t+f_l(u)_x=0$) 
$\zeta_{t_x,-}$, $\zeta_{t_y,+}$, issuing from $(0,t_x)$ and 
 from $(0,t_y)$, respectively, given by
 \vspace{-3pt}
\begin{equation*}
\begin{aligned}
 \zeta_{t_x,-}(t)&= f'_l(u_l(t_x-))\cdot (t-t_x)=f'_l\big(\pi_{l,+}^r(\omega(x+))\big)\cdot (t-t_x)
 \qquad t \in [0,t_x],\\
 \noalign{\smallskip}
 \zeta_{t_y,+}(t)&= f'_l(u_l(t_y+))\cdot (t-t_y)=f'_l\big(\pi_{l,+}^r(\omega(y-))\big)\cdot (t-t_y)
 \qquad t\in [0,t_y,]\,.
 \end{aligned}
  \end{equation*}
  \vspace{-3pt}
  By definitions~\eqref{pimap-def},  \eqref{phi1-def}, , \eqref{psi1-def}, we find that
\begin{equation*}
\begin{aligned}
 \zeta_{t_x,-}(0)&=-f'_l\big(\pi_{l,+}^r(\omega(x+))\big)\cdot \big(T-{x}/{f'_r(\omega(x+))}\big)=
 -f'_l\circ f_{l,+}^{-1}\circ f_r(\omega(x+))\cdot \big(T-{x}/{f'_r(\omega(x+))}\big)=\phi_1(x+)\,,
 \\
 \noalign{\smallskip}
 \zeta_{t_y,+}(0)&=
 -f'_l\big(\pi_{l,+}^r(\omega(y-))\big)\cdot \big(T-{y}/{f'_r(\omega(y-))}\big)=
 -f'_l\circ f_{l,+}^{-1}\circ f_r(\omega(y-))\cdot \big(T-{y}/{f'_r(\omega(y-))}\big)=\phi_1(y-)\,.
  \end{aligned}
  \end{equation*}
Since $t_x>t_y$ and because backward characteristics 
cannot cross on $(-\infty,0)\times(0,+\infty)$, it follows that
$\phi_1(x+)= \zeta_{t_x,-}(0)\leq \zeta_{t_y,+}(0)=\phi_1(y-)$,
proving the nondecreasing monotonicity of $\varphi_1$.
This completes the proof that  $\omega\in \mathcal{A}_1(T)$
in the case $(\omega(0-), \omega(0+))\in\mathcal{T}_1$.
\smallskip
 
\noindent
{\bf Case 4.} $(\omega(0-), \omega(0+))\in\mathcal{T}_2$.\\
With entirely similar arguments to {Case 3}, we deduce that $\omega \in \mathcal{A}_2(T)$.
\smallskip

\noindent
{\bf Case 5.} $(\omega(0-), \omega(0+))\in\mathcal{T}_{3,-}\cup \mathcal{T}_{3,+}$.\\
We assume that $\omega(0-) >\theta_l$, $\omega(0+)>\theta_r$. The cases with $\omega(0-) =\theta_l$,
or  $\omega(0+)=\theta_r$, can be treated with entirely similar arguments, relying on the analysis of Case 2.
Let $\xi_{0,-}, \xi_{0,+}$ be the minimal and maximal backward characteristics through $(0,T)$. 
Then we have $\xi_{0,-}(0)=-f'_l(\omega(0-))\cdot T<0<\xi_{0,+}(0)=-f'_r(\omega(0+))\cdot T$.
Since backward characteristics starting at points $(x,T)$
with $x<0$ or $x>0$ cannot cross $\xi_{0,-}$ and $\xi_{0,+}$, respectively,
we deduce that $f'_l(\omega(x\pm))\geq  \frac{x}{T}+f'_l(\omega(0-))$
for all $x<0$ and $f'_r(\omega(x\pm))\leq  \frac{x}{T}+f'_r(\omega(0+))$
for all $x>0$. Thus, setting $L=R=0$, the conditions~\eqref{c71}
are satisfied. Moreover, with the same arguments of Case 1 we deduce that the map  $\varphi_3$
in~\eqref{phi3-def} is nondecreasing. Hence, we have shown that $\omega\in \mathcal{A}_{3}^{AB}(T)$,
and this completes the proof of $\mathcal{A}(T)\subseteq \mathcal{A}_{1}(T)\cup\mathcal{A}_{2}(T)\cup\mathcal{A}_{3}^{AB}(T)$.
\qed

\subsection{Proof of \ ${\bf \mathcal{A}_{1}(T)\cup\mathcal{A}_{2}(T)\cup\mathcal{A}_{3}^{AB}(T)\subseteq
\mathcal{A}(T)}$.}
Given a function $\omega\in \mathcal{A}_{1}(T)\cup\mathcal{A}_{2}(T)\cup\mathcal{A}_{3}^{AB}(T)$,
we will show that there exists an initial datum $\overline{u}\in {\bf L}^{\infty}(\mathbb{R})$ such that $\mathcal{S}^{AB}_{T} \overline{u}=\omega$.
We shall analyze only two cases, the others being entirely similar.

\noindent
{\bf Case 1.} $\omega\in \mathcal{A}_2(T)$.\\
We assume that $\omega(0-) > \pi_{l,-}^r(\omega(0+))$, the case $\omega(0-) =\pi_{l,-}^r(\omega(0+))$
being entirely similar and simpler.
Hence, we have $\pi_{r,-}^l(\omega(0-))>\omega(0+)$.
We will construct the initial datum $\overline u$  with the desired property 
adopting a similar procedure to~\cite{MR1616586}, which consists of the following steps: 
\begin{itemize}
\item[1.] For every $x\neq 0$ we trace the lines $\vartheta_{x,-}$,$\vartheta_{x,+}$ through $(T,x)$ with slope $f'_{l}(\omega(x-))$, $f'_{l}(\omega(x+))$, respectively, if $x< 0$, and $f'_{r}(\omega(x-))$, $f'_{r}(\omega(x+))$, respectively, if $x>0$. 
At $x=0$ we trace the lines $\vartheta_{0,-}$, $\vartheta_{0,+}$  through $(T,0)$ with slope 
$f'_r(\pi_{r,-}^l(\omega(0-)))$, $f'_r(\omega(0+)$, respectively.
Because of~\eqref{c4}, $\vartheta_L^-$ and all lines $\{\vartheta_{x,\pm}: \, x\geq 0\ \,\text{or} \ x<L\}$ 
 reach the $x$-axis without crossing the
line $x=0$ at times $t>0$, while $\vartheta_{L,+}$ and all lines $\{\vartheta_{x,\pm}: \, L<x<0\}$ cross the 
line $x=0$  at a time $t\geq 0$. Then, we redefine $\vartheta_{L,+}$ and  $\{\vartheta_{x,\pm}: \, L<x<0\}$
as polygonal lines that, after crossing $x=0$, continue  with slope 
$f'_r(\pi_{r,-}^l(\omega(L+)))$ and
$f'_r(\pi_{r,-}^l(\omega(x-)))$, $f'_r(\pi_{r,-}^l(\omega(x+))$, respectively.
Since the curves  $\vartheta_{x,\pm}$ are defined so that one has
$\vartheta_{x,\pm}(0)=\varphi_2(x\pm)$ for all $x$,
from the monotonicity of the map $\varphi_2$ in~\eqref{phi2-def} we deduce that $\vartheta_{x,\pm}$
never intersect each other in the region $\mathbb{R}\times (0,T)$.
We will treat the polygonal lines $\vartheta_{x,\pm}$, $x\in\mathbb{R}$, as 
(minimal and maximal) backward characteristics of the $AB$-entropy
solution that we are constructing on $\mathbb{R}\times [0,T]$.

\item[2.] Since the solution is constant along genuine characteristics, 
for every $x\in (-\infty, \vartheta_{L,-}(0))\cup (\vartheta_{0,+}(0), +\infty)$
such that $x=\vartheta_{y,\pm}(0)$ for some $y\in(-\infty, L)\cup (0,+\infty)$, we will set $\overline u(x)=\omega(y\pm)$,
while for for every $x\in  (\vartheta_{L,+}(0), \vartheta_{0,-}(0))$
such that $x=\vartheta_{y,\pm}(0)$ for some $y\in (L,0)$, we will set $\overline u(x)=\pi_{r,-}^l(\omega(y\pm))$.
The set of remaining $x$ is a disjoint union of countably many open intervals, say $(x^-_n, x^+_n)$,
$n\in\mathbb{N}$,
with $x^-_n=\vartheta_{y_n,-}(0), x^+_n=\vartheta_{y_n,+}(0)$, for some $y_n\in\mathbb{R}$, 
where $\overline u$ is defined so to produce a compression wave which generates a discontinuity 
at the point $(y_n,T)$.

\item[3.] According with the definition of $\overline u$ in step 2,  we define a function $u:\mathbb{R}\times [0,T]\rightarrow \mathbb{R}$ which is constant along the lines $\vartheta_{x,\pm}$ that do not cross $x=0$, and it is
piecewise constant along the polygonal lines $\vartheta_{x,\pm}$ that intersect $x=0$,
changing value at $x=0$ so to satisfy the interface entropy condition~\eqref{int-entr-cond3}. Namely, we set $u$
equal to $\omega(y\pm)$ along the line $\vartheta_{y,\pm}(t), t \in [0,T]$, 
when $y\in(-\infty, L)\cup (0,+\infty)$, and along the segment of polygonal $\vartheta_{y,\pm}(t), t \in [\tau_y, T]$, with $\vartheta_{y,\pm}(\tau_y)=0$, when
$y\in (L,0)$. Instead we define $u$ as $\pi_{r,-}^l(\omega(y\pm))$ 
along the segment of polygonal $\vartheta_{y,\pm}(t), t \in [0,\tau_y]$, with $\vartheta_{y,\pm}(\tau_y)=0$,
when $y\in (L,0)$. Finally, for any $x\in (x^-_n, x^+_n)$
we let $u$ to be equal to $\overline u(x)$ on the right of $x=0$
and to be equal to $\pi_{l,-}^r(\overline u(x))$ on the left of $x=0$, along a polygonal line $\eta_x(t) $, $t\in [0,T]$,
which connects $(x,0)$ with $(y_n,T)$.

\item[4.] With the same arguments of~\cite{MR1616586} one can show that the function $u$ constructed in step 3: is locally Lipschitz continuous on $\mathbb{R}\times [0,T]$; 
it is a classical solution of $u_t+f_l(u)_x$ on $(-\infty, 0)\times (0,T)$,
and of $u_t+f_r(u)_x$ on $(0, +\infty)\times (0,T)$;
it is continuous with respect to the the ${\bf L}^1_{loc}$
topology as a function from $[0,T]$ to ${\bf L^\infty}(\mathbb{R})$;  it attains the initial data $\overline u$ 
at time $t=0$ and the terminal profile $\omega$ at time $t=T$. Moreover, $u$
satisfies the interface entropy condition~\eqref{int-entr-cond3} associated to the connection $AB$.
 \end{itemize}
1. For each $x\neq 0, L$, consider the polygonal lines 
\begin{equation}
\vartheta_{x,\pm}(t):=
\begin{cases}
x+f'_l(\omega(x\pm))(t-T) &\text{  if  } \ x<L,\  t\in [0,T],\\
\noalign{\smallskip}
x+f'_l(\omega(x\pm))(t-T) &\text{  if  } \ L<x<0,\  t\in \big[T-{x}/{f'_l(\omega(x\pm))},\,T\big],\\
\noalign{\smallskip}
f'_r(\pi_{r,-}^l(\omega(x\pm)))\big(t-T+{x}/{f'_l(\omega(x\pm))}\big) &\text{  if  } \ L<x<0,\  t\in \big[0,\,T-{x}/{f'_l(\omega(x\pm))}\big],\\
\noalign{\smallskip}
x+f'_r(\omega(x\pm))(t-T) &\text{  if  }\  x>0,\  t\in [0,T],
\end{cases}\label{chamin}
\end{equation}
and, at $x=0$, $x=L$, set
\begin{equation}
\label{chamin2}
\begin{aligned}
\vartheta_{0,-}(t)&:= f'_r (\pi_{r,-}^l(\omega(0-))) (t-T)\qquad \text{if} \qquad t\in [0,T],
\\
\vartheta_{0,+}(t)&:= f'_r ((\omega(0+)) (t-T)\qquad\qquad\, \text{if}\quad\quad\, t\in [0,T],
\\
\vartheta_{L,-}(t)&:= L+f'_l(\omega(L-))(t-T)\qquad \ \, \text{if}\quad\quad\, t\in [0,T],
\\
\vartheta_{L,+}(t)&:=
\begin{cases}
x+f'_l(\omega(L+))(t-T)
 &\text{  if  } \  t\in \big[T-{x}/{f'_l(\omega(L+))},\,T\big],\\
\noalign{\smallskip}
f'_r(\pi_{r,-}^l(\omega(L+)))\big(t-T+{x}/{f'_l(\omega(L+))}\big) &\text{  if  }\  t\in \big[0,\,T-{x}/{f'_l(\omega(L+))}\big].
\end{cases}
\end{aligned}
\end{equation}
Notice that, by definitions~\eqref{pimap-def}, \eqref{phi2-def}, \eqref{psi2-def},
we have $\vartheta_{x,\pm}(0)=\varphi_2(x\pm)$ for all $x$,
and $\vartheta_{x,\pm}(\psi_2(x\pm))=0$ for all $x\in (L,0)$.
Then, relying on~\eqref{c4}, on the nondecreasing monotonicity of $\varphi_2$, 
and on the increasing monotonicity of $\psi_2$, we deduce that the polygonal lines
$\vartheta_{x,\pm}$, $x\in\mathbb{R}$,
never intersect each other in the region $\mathbb{R}\times (0,T)$.\\

\noindent
2. Consider the following partition of $\mathbb{R}$ (see Fig.$\ref{Partizione}$):
\begin{equation}
\label{partition}
\begin{aligned}
\mathcal{I}_{R} &\doteq \big\lbrace x\in \mathbb{R}:\ \ \exists~y<z:\ \vartheta_{y,+}(0)=\vartheta_{z,-}(0)=x  \big\rbrace,\\
\mathcal{I}_C &\doteq \big\lbrace x\in \mathbb{R}: \ \ \nexists~y \in\mathbb{R} :\ \vartheta_{y,-}(0)=x\ \, \text{ or } \,\vartheta_{y,+}(0)=x
 \big\rbrace,\\
 \mathcal{I}_W &\doteq \big\lbrace x\in \mathbb{R}:\ \ \exists !~y:\ \vartheta_{y,-}(0)=x\ \, \text{ or }\, \vartheta_{y,+}(0)=x \big\rbrace.
\end{aligned}
\end{equation}

\vspace{-15pt}
\begin{figure}[!htbp] 
\begin{center}

\begin{tikzpicture}[scale=2,cap=round]

 \def\costhirty{0.8660256}

\tikzstyle{axes}=[]
  \tikzstyle{important line}=[very thick]
  \tikzstyle{information text}=[rounded corners,fill=red!10,inner sep=1ex]

\draw[thick,->] (-2.2,0) -- (2.2,0) node[right] {$x$};
\draw[thick,->] (0,0) -- (0,2) node[right] {$t$};

\draw [domain=-2.2:2.2] plot (\x, {1.7}); 
\draw (2,1.7) node[anchor=south] {$\omega$};
\draw  (0,0) node[anchor=north] {$\mathcal{I}_W$};
\draw [dashed](-0.85,1.7) -- (0,0.65) -- (0.38,0);
\draw [dashed](-1,1.7) -- (0,0.48) -- (0.26,0);
\draw [dashed](-1.15,1.7) -- (0,0.30) -- (0.14,0);
\draw [dashed](-1.3,1.7) -- (0,0.12) -- (0.02,0);
\draw [dashed](-1.45,1.7) -- (-0.1,0);
\draw [dashed](-1.60,1.7) -- (-0.22,0);
\draw [dashed](-1.75,1.7) -- (-0.34,0);
\draw (-0.7,1.7) node[anchor=south] {$y_{n-1}$};
\draw (-0.7,1.7) -- (0,1.1) -- (0.9,0);
\draw [dashed](-0.7,1.7) -- (0,1) -- (0.77,0);
\draw [dashed](-0.7,1.7) -- (0,0.9) -- (0.63,0);
\draw  (0.72,0) node[anchor=north] {$\mathcal{I}_C^{n-1}$};
\draw (-0.7,1.7) -- (0,0.8) -- (0.5,0);
\draw (0.5,1.7) -- (0.9,0);
\draw [dashed](-0.52,1.7) -- (0,1.24) -- (0.9,0);
\draw [dashed](-0.32,1.7) -- (0,1.4) -- (0.9,0);
\draw [dashed](-0.135,1.7) -- (0,1.55) -- (0.9,0);

\draw [dashed](0,1.7) -- (0.9,0);
\draw [dashed](0.15,1.7) -- (0.9,0);
\draw [dashed](0.3,1.7) -- (0.9,0);
\draw [dashed](0.42,1.7) -- (0.9,0);
\draw [dashed](0.64,1.7) -- (0.95,0);
\draw [dashed](0.79,1.7) -- (1,0);
\draw  (1.13,0) node[anchor=north] {$\mathcal{I}_W$};
\draw [dashed](0.94,1.7) -- (1.05,0);
\draw [dashed](1.09,1.7) -- (1.1,0);
\draw [dashed](1.22,1.7) -- (1.15,0);

\draw (1.37,1.7) node[anchor=south] {$y_n$};
\draw (1.37,1.7) -- (1.2,0);
\draw (1.37,1.7) -- (1.6,0);

\draw [dashed](1.37,1.7) -- (1.3,0);
\draw [dashed](1.37,1.7) -- (1.4,0);
\draw [dashed](1.37,1.7) -- (1.5,0);
\draw  (1.4,0) node[anchor=north] {$\mathcal{I}_C^n$};

\end{tikzpicture}
\caption{An example of partition of $\mathbb{R}$ associated to the profile $\omega$}
\label{Partizione}
\end{center}
\end{figure}
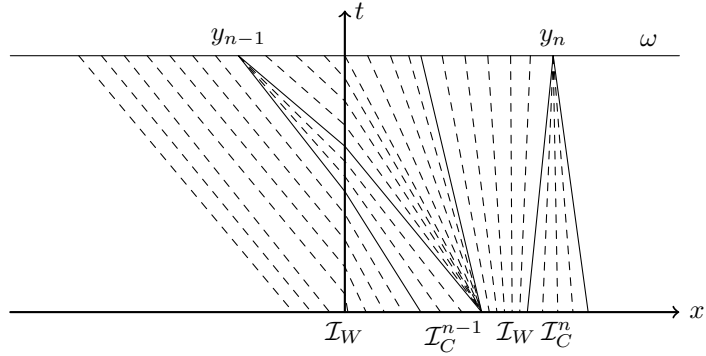

\vspace{-10pt}
\noindent
Some considerations about this partition are useful for the next. The set $\mathcal{I}_R$ consists of the
centres of rarefaction waves originated at time $t=0$, the set $\mathcal{I}_C$ consists of the 
starting points of the compression
waves that generate shocks at time $T$, and $\mathcal{I}_W$ collects the starting points of all other 
waves. With entirely similar arguments to~\cite{MR1616586}, one can verify that:  
\vspace{-5pt}
\begin{itemize}
\item[-] The set $\mathcal{I}_R$ contains at most countably many points;
\item[-] The set $\mathcal{I}_C$ is a disjoint union of at most countably many open intervals 
of the form 
\vspace{-5pt}
\begin{equation}
\begin{aligned}
\mathcal{I}^n&=(x_n^-,x_n^+), \quad x_n^\pm=\vartheta_{y_n,\pm}(0),
\qquad y_n\in (-\infty, L)\cup [0,+\infty),
\\
\mathcal{I}^n_L&=(x_n^-,x_n^+), \quad x_n^\pm=\vartheta_{y_n,\pm}(0),
\qquad y_n\in [L, 0)\,,
\end{aligned}
\end{equation}
\vspace{-13pt}

\noindent
with $y_n$ point of discontinuity of $\omega$.
Notice that, since $\vartheta_{y_n,\pm}(0)=\varphi_2(y_n\pm)$,
by the monotonicity of $\varphi_2$ and $f_l, f'_r$, it follows that $\omega(y_n-)>\omega(y_n+)$
for all $y_n\neq 0$. Moreover, we observed at the beginning that we have $\pi_{r,-}^l(\omega(0-))>\omega(0+)$.
Thus, we will construct compression waves generating a shock connecting the states $\omega(y_n-), \omega(y_n+)$ at $(y_n,T), y_n\neq 0$,
and connecting the states $\pi_{r,-}^l(\omega(0-)), \omega(0+)$ at $(0,T)$.
\end{itemize}
In order to define the initial data in the sets $\mathcal{I}^n_L$, for any $(x_n^-,x_n^+)$
with $x_n^\pm=\vartheta_{y_n,\pm}(0), \ L< y_n<0$, 
setting $\alpha_n^\pm\doteq f'_r(\pi_{r,-}^l(\omega(y_n\pm)))$, consider 
the function
\vspace{-1pt}
\begin{equation}
h_n(x,\alpha)=T-{y_n}/\big[{f'_l\circ \pi_{l,-}^r\circ (f'_r)^{-1}(\alpha)}\big]+{x}/{\alpha}
\qquad\quad x\in (x_n^-,x_n^+),\quad  \alpha\in [\alpha_n^+, \alpha_n^-]\,.
\end{equation}
\vspace{-10pt}

\noindent
Notice that, because of the monotonicity of $f'_r$, $\pi_{r,-}^l$,  and since 
by~\eqref{c21} we have
$\omega(y_n+)<\omega(y_n-)<\theta_l$,
it follows that $\alpha_n^+<\alpha_n^-<0$. Moreover, letting $\tau_n^\pm$ be the times of intersection of 
$\vartheta_{y_n,\pm}$ with $x=0$, i.e., such that $\vartheta_{y_n,\pm}(\tau_n^\pm)=0$, 
we have 
\vspace{-2pt}
\begin{equation}
\tau_n^\pm=T-{y_n}/{f'_l(\omega(y_n^\pm))}=
T-{y_n}/\big[{f'_l\circ \pi_{l,-}^r\circ (f'_r)^{-1}(\alpha_n^\pm)}\big]={-x_n^\pm}/{\alpha_n^\pm}\,.
\end{equation}
Then, by a direct computation one finds that, for any $x\in (x_n^-,x_n^+)$, there holds
\begin{equation}
\begin{aligned}
&\hspace{1.3in}
h_n(x,\alpha_n^+)={(x-x_n^+)}/{\alpha_n^+}>0,\qquad\qquad h_n(x,\alpha_n^-)={(x-x_n^-)}/{\alpha_n^-}<0\,,\\
\noalign{\smallskip}
&\partial_{\alpha}h_n(x,\alpha)=y_n\dfrac{\alpha \cdot f''_l\circ \pi_{l,-}^r\circ (f'_l)^{-1}(\alpha)}{\big[f'_l\circ \pi_{l,-}^r\circ (f'_r)^{-1}(\alpha)\big]^2\cdot \big[f'_l\circ \pi_{l,-}^r\circ (f'_r)^{-1}(\alpha)\big]\cdot 
\big[f''_r\circ (f'_r)^{-1}(\alpha)\big]}-\dfrac{x}{\alpha^2}<0
\qquad\forall~\alpha\in(\alpha_n^+, \alpha_n^-)\,.
\end{aligned}
\end{equation}
Hence, we may define a continuous, decreasing map $\alpha_n: (x_n^-, x_n^+) \to (\alpha_n^+,\alpha_n^-)$ that satisfies
\begin{equation}
\label{crossing-cond-charact}
T-{y_n}/\big[{f'_l\circ \pi_{l,-}^r\circ (f'_r)^{-1}(\alpha_n(x))}\big]=-{x}/{\alpha_n(x)}
\qquad\quad x\in (x_n^-, x_n^+)\,.
\end{equation}
Notice that $\lim_{x\to x_n^{\pm}}\alpha_n(x)=\alpha_n^{\pm}$.
The quantity $\alpha_n(x)$ determines the slope $\lambda_n^+$ on the right of $x=0$ of a polygonal line $\eta_x$,
connecting $(x,0)$ and $(y_n,T)$, with the property that, letting $\lambda_n^-\doteq f'_l\circ \pi_{l,-}^r\circ (f'_r)^{-1}(\alpha_n(x))$
be the slope of $\eta_x$ on the left of $x=0$, there holds
\vspace{-5pt}
\begin{equation}
(f'_l)^{-1}(\lambda_n^-)=
\pi_{l,-}^r\big((f'_r)^{-1}(\lambda_n^+)\big)\,,
\end{equation}
which guarantees that the states $u_l\doteq (f'_l)^{-1}(\lambda_n^-)$,
$u_r\doteq (f'_l)^{-1}(\lambda_n^+)$, satisfy the  interface entropy condition~\eqref{int-entr-cond3}.
In the case $(x_n^-, x_n^+)\subset \mathcal{I}_L^n$ is of the form 
$x_n^\pm=\vartheta_{L,\pm}(0)$, $\vartheta_{L,-}(0)<0<\vartheta_{L,+}(0)$, with the same arguments 
of above we may define a continuous, decreasing function 
$\alpha_n: [0, x_n^+) \to (\alpha_n^+,y_n/T]$ that satisfies
 the equalities in~\eqref{crossing-cond-charact} for all $x\in[0,x_n^+)$,
 and there holds $\alpha_n(0)=y_n/T$.
Then, we define the initial data as
\begin{equation}\label{initial control}
\overline {u}(x):=\begin{cases}
\omega(y\pm)\ \ &\text{ if }\quad x\in \mathcal{I}_W,\ \
x=\vartheta_{y,\pm}(0),\ \  y\in(-\infty, L)\cup (0,+\infty),\\
\noalign{\smallskip}
\pi_{r,-}^l(\omega(0-))\ \ &\text{ if }\quad x\in \mathcal{I}_W,\ \ x=\vartheta_{0,-}(0),\\
\noalign{\smallskip}
\omega(0+)\ \ &\text{ if }\quad x\in \mathcal{I}_W,\ \ x=\vartheta_{0,+}(0),\\
\noalign{\smallskip}
\omega(L-)\ \ &\text{ if }\quad x\in \mathcal{I}_W,\ \ x=\vartheta_{L,-}(0),\\
\noalign{\smallskip}
\pi_{r,-}^l(\omega(L+))\ \ &\text{ if }\quad x\in \mathcal{I}_W,\ \ x=\vartheta_{L,+}(0),\\
\noalign{\smallskip}
\pi_{r,-}^l(\omega(y\pm))\ \ &\text{ if }\quad x\in \mathcal{I}_W,\ \
x=\vartheta_{y,\pm}(0),\ \  y\in(L,0),\\
\noalign{\smallskip}
(f'_l)^{-1}\big({(y_n-x)}/{T}\big)\ \ &\text{ if }\quad x\in (x_n^-,x_n^+)\subseteq \mathcal{I}_C,\ \
x_n^\pm=\vartheta_{y_n,\pm}(0),\ \ y_n<L,
\\
\noalign{\smallskip}
(f'_r)^{-1}\big({(y_n-x)}/{T}\big)\ \ &\text{ if }\quad x\in (x_n^-,x_n^+)\subseteq \mathcal{I}_C,\ \
x_n^\pm=\vartheta_{y_n,\pm}(0),\ \ y_n\geq 0,
\\
\noalign{\smallskip}
(f'_r)^{-1}(\alpha_n(x))\quad &\text{ if }\quad x\in (x_n^-,x_n^+)\subseteq \mathcal{I}_C,\ \
x_n^\pm=\vartheta_{y_n,\pm}(0),\ \ L<y_n< 0,\\
\noalign{\smallskip}
(f'_r)^{-1}(\alpha_n(x))\quad &\text{ if }\quad x\in (x_n^-,x_n^+)\subseteq \mathcal{I}_C,\ \
x_n^\pm=\vartheta_{L,\pm}(0),\ \ \  x\geq 0,\\
\noalign{\smallskip}
(f'_l)^{-1}\big({(y_n-x)}/{T}\big)\ \ &\text{ if }\quad x\in (x_n^-,x_n^+)\subseteq \mathcal{I}_C,\ \
x_n^\pm=\vartheta_{L,\pm}(0),\ \ \  x< 0.
\end{cases}
\end{equation}
Notice that $\overline u$ is not defined on the set $\mathcal{I}_R$ which is of measure zero since 
it is countable. Moreover, we have 
\begin{equation}
\label{inital-data-linf-bound}
|\overline u(x)|\leq  M \doteq
\sup\big\{\!\max\{|\omega(x)|, |\pi_{r,-}^l (\omega (x))|\};\, x\in\mathbb{R}\big\}\,.
\end{equation}

\noindent
3. In order to define the solution $u$ in the region of compression waves, 
for any $x\in\mathcal{I}_C$, consider the polygonal lines 
\vspace{-5pt}
\begin{equation}
\label{chamin3}
\eta_x(t)\!:=\!
\begin{cases}
x+{\big((y_n-x)\, t\big)}/{T}\quad&\text{if}\quad\ x\in (x_n^-,x_n^+)\!\subseteq\! \mathcal{I}_C,\ \
x_n^\pm=\vartheta_{y_n,\pm}(0),\ \ y_n\!<\!L\ \text{or}\ y_n\!\geq\! 0,
\ \ t\in [0,T],\\
\noalign{\smallskip}
x+\alpha_n(x)\, t\quad&\text{if}\quad\ x\in (x_n^-,x_n^+)\!\subseteq\! \mathcal{I}_C,\ \
x_n^\pm=\vartheta_{y_n,\pm}(0),\ \ L\!<y_n\!<0,
\ \ t\in [0,-\frac{x}{\alpha_n(x)}],\\
\noalign{\smallskip}
x\!+\!f'_l\circ \pi_{l,-}^r\!\circ\! (f'_r)^{-1}(\alpha_n(x))\,t\ &\text{if}\quad\ x\in (x_n^-,x_n^+)\!\subseteq \!\mathcal{I}_C,\ \
x_n^\pm\!=\!\vartheta_{y_n,\pm}(0),\ \ L\!<y_n\!<0,
\ \  t\in [-\frac{x}{\alpha_n(x)},T],\\
\noalign{\smallskip}
x+\alpha_n(x)\, t\quad&\text{if}\quad\ x\in (x_n^-,x_n^+)\!\subseteq\! \mathcal{I}_C,\ \
x_n^\pm=\vartheta_{L,\pm}(0),\ \ x\geq 0, 
\ \ t\in [0,-\frac{x}{\alpha_n(x)}],\\
\noalign{\smallskip}
x\!+\!f'_l\circ \pi_{l,-}^r\!\circ\! (f'_r)^{-1}(\alpha_n(x))\,t\ &\text{if}\quad\ x\in (x_n^-,x_n^+)\!\subseteq \!\mathcal{I}_C,\ \
x_n^\pm\!=\!\vartheta_{L,\pm}(0),\ \ x\geq 0, 
\ \  t\in [-\frac{x}{\alpha_n(x)},T],\\
\noalign{\smallskip}
x+{\big((y_n-x)\, t\big)}/{T}\quad&\text{if}\quad\ x\in (x_n^-,x_n^+)\!\subseteq\! \mathcal{I}_C,\ \
x_n^\pm=\vartheta_{L,\pm}(0),\ \ x< 0, 
\ \ t\in [0,T].
\end{cases}
\end{equation}
Observe that, by construction the polygonal lines $\vartheta_{x,\pm}$, $x\in\mathbb{R}$
in~\eqref{chamin}-\eqref{chamin2}, and $\eta_x, x\in \mathcal{I}_C$ in~\eqref{chamin3},
never intersect each other in the region $\mathbb{R}\times (0,T)$
and there holds
\begin{equation}
\forall~(x,t)\in\mathbb{R}\times (0,T) \quad \exists !~y\in \mathbb{R} \qquad \text{s.t.}\quad 
x=\vartheta_{y,-}(t),\quad \text{or}\quad x=\vartheta_{y,+}(t)\quad \text{or}\quad x=\eta_y(t), \ y\in \mathcal{I}_C\,.
\end{equation}
Thus, we may define  on $(\mathbb{R}\setminus\{0\})\times (0,T)$ the function:
\begin{equation}\label{soluzione-costruita}
u(x,t):=
 \begin{cases}
 \omega(y\pm)&\text{ if }\quad \exists~y \in(-\infty,L)\cup(0,+\infty):\ x={\vartheta}_{y,\pm}(t),\\
 \noalign{\smallskip}
  \omega(y\pm)&\text{ if }\quad  \exists~y \in[L,0):\ x=\vartheta_{y,\pm}(t)<0,\\
 \noalign{\smallskip}
  \pi_{r,-}^l(\omega(y\pm))&\text{ if }\quad  \exists~y \in[L,0):\ x=\vartheta_{y,\pm}(t)>0,\\
   \noalign{\smallskip}
  \pi_{r,-}^l(\omega(0-))&\text{ if }\quad  x=\vartheta_{0,-}(t),\\
     \noalign{\smallskip}
  \omega(0+)&\text{ if }\quad  x=\vartheta_{0,+}(t),\\
 \noalign{\smallskip}
  \overline u(y)&\text{ if }\quad \exists~y\in \mathcal{I}_L^n:\ x={\eta}_{y}(t)>0,\\
 \noalign{\smallskip}
   \pi_{l,-}^r(\overline u(y))&\text{ if }\quad \exists~y\in \mathcal{I}_L^n:\ x={\eta}_{y}(t)<0,\\
 \noalign{\smallskip}
  \overline u(y)&\text{ if }\quad \exists~y\in \mathcal{I}^n:\ x={\eta}_{y}(t)\,.
 \end{cases}.
\end{equation}

\noindent
4. 
By construction the function $u$ in~\eqref{soluzione-costruita} is continuous on  $\mathbb{R}\times (0,T)$
and satisfies the  interface entropy condition~\eqref{int-entr-cond3} at $x=0$.
Moreover, with the same type of analysis in~\cite{MR1616586} one can show that there holds
\begin{equation}
    D^-_x u(x,t)\geq\big[{f''_{l,r}(u(x,t))\cdot (t-T)}\big]^{-1}\geq\big[{c\cdot (t-T)}\big]^{-1}\qquad\forall (x,t)\in \mathbb{R}\times (0,T)\,.
\end{equation}
On the other hand, relying on~\eqref{inital-data-linf-bound}, \eqref{soluzione-costruita}, 
and on the assumption {\bf H1)}, 
with the same arguments of the proof
of $\mathcal{A}(T)\subseteq \mathcal{A}_{1}(T)\cup\mathcal{A}_{2}(T)\cup\mathcal{A}_{3}^{AB}(T)$  we derive
\vspace{-5pt}
\begin{equation}
 D^+_x u(x,t)\leq 
 \begin{cases}
  \big[{f''_{l}(u(x,t))\cdot t}\big]^{-1}\leq  \big[{c\cdot t}\big]^{-1}  &\forall~x< \eta_0(t)\,,\\
  \noalign{\medskip}
     \quad\ \dfrac{f'_l(u(x,t))}{f''_l(u(x,t))\cdot x }\leq  {M'}\big[{c\cdot x }\big]^{-1} \qquad &\forall~\eta_0(t)<x<0,
    \\
     \noalign{\medskip}
 \big[{f''_{r}(u(x,t))\cdot t}\big]^{-1} \leq  \big[{c\cdot t}\big]^{-1}&\forall~x>0\,.\\
  \end{cases}
\end{equation}
for some constant $M'>0$. Hence $u$ is locally Lipschitz continuous
and therefore it  is differentiable almost everywhere.
By a direct computation one can check that $u$ is a classical solution of $u_t+f_l(u)_x$ on $(-\infty, 0)\times (0,T)$,
and of $u_t+f_r(u)_x$ on $(0, +\infty)\times (0,T)$. 
Hence, $u$ is an $AB$-entropy solution of~\eqref{eq1},\eqref{flux}. 
Finally, with the same arguments in~\cite{MR1616586}
one verifies the continuity of $t\to u(\cdot, t)$ on $[0,T]$ with respect to the ${\bf L}^1_{loc}$-topology,
and that $u(\cdot, 0)=\overline u,$ $u(\cdot, T)=\omega$, which proves that $\omega=S_T^{AB}\,\overline u
\in \mathcal{A}(T)$.\\
%
%

\noindent
{\bf Case 2.}   $\omega\in \mathcal{A}_{3}^{AB}(T)$, \ $L=0=R$, \ $(\omega(0-),\omega(0+))\in \mathcal{T}_{3,-}$.\\
Since $(\omega(0-),\omega(0+))\in \mathcal{T}_{3,-}$ it follows that $\omega(0-)\geq \theta_l$
and $f_l(\omega(0-))\leq f_r(\omega(0+))$.
We assume that $\omega(0-)>\theta_l$, 
and that $f_l(\omega(0-))<f_r(\omega(0+))$, the cases with  $\omega(0-)=\theta_l$ 
or with $f_l(\omega(0-))= f_r(\omega(0+))$ being entirely similar.
We follow the same procedure of the previous case discussing only the points where 
there is a difference in the construction of the initial data $\overline u$ and of the solution $u$.
\smallskip

\noindent
1. For each $x\neq 0$, consider the lines 
\vspace{-8pt}
\begin{equation}
\vartheta_{x,\pm}(t):=
\begin{cases}
x+f'_l(\omega(x\pm))(t-T) &\text{  if  } \ \ x<0,\  t\in [0,T],\\
\noalign{\smallskip}
x+f'_r(\omega(x\pm))(t-T) &\text{  if  }\  \  x>0,\  t\in [0,T],
\end{cases}
\label{chamin31}
\end{equation}

\vspace{-10pt}
\noindent
and, for $x=0$, set
\begin{equation}
\label{chamin32}
\begin{aligned}
&\vartheta_{0,-}(t):=f'_l(\omega(0-))(t-T), \qquad \quad
\vartheta_{0,+}(t):=
f'_r(\omega(0+))(t-T),\\
&\qquad\qquad\qquad
\vartheta_{0,*}(t):=f'_r(\pi_{r,-}^l(\omega(0-)))(t-T),
\end{aligned}
 \qquad\quad \forall~ t\in [0,T].
\end{equation}

\noindent
2. Then, letting $x_0^\pm\doteq \vartheta_{0,\pm}(0)$, $x_0^*\doteq \vartheta_{0,*}(0)$,
consider the  partition of $\mathbb{R}\setminus\{0\}$:
\vspace{-5pt}
\begin{equation}
\label{partition2}
\begin{aligned}
\mathcal{I}_{R} &\doteq \big\lbrace x\in \mathbb{R}:\ \ \exists~y<z:\ \vartheta_{y,+}(0)=\vartheta_{z,-}(0)=x  \big\rbrace,\\
\mathcal{I}_C &\doteq \big\lbrace x\in \mathbb{R}\setminus [x_0^-, x^*]: \ \ \nexists~y \in\mathbb{R} :\ \vartheta_{y,-}(0)=x\ \, \text{ or } \,\vartheta_{y,+}(0)=x
 \big\rbrace,\\
 \mathcal{I}_W &\doteq \big\lbrace x\in \mathbb{R}:\ \ \exists !~y:\ \vartheta_{y,-}(0)=x\ \, \text{ or }\, \vartheta_{y,+}(0)=x \big\rbrace,\\
  \mathcal{I}_{0,-}&\doteq(x_0^-,0),\qquad \mathcal{I}_{0,+}\doteq (0,x_0^*\,)\,.
\end{aligned}
\end{equation}
Here $ \mathcal{I}_{0,-},  \mathcal{I}_{0,+}$ are intervals where the initial data $\overline u$
will assume the constant value 
$\omega(0-)$ and $\pi_{r,-}^l(\omega(0-))$, respectively, 
while $\mathcal{I}_C$
 is a disjoint union of at most countably many open intervals 
of the form 
\vspace{-5pt}
\begin{equation}
\begin{aligned}
\mathcal{I}^n &=(x_n^-,x_n^+), \quad x_n^\pm=\vartheta_{y_n,\pm}(0),
\qquad y_n\in (-\infty, 0)\cup (0,+\infty),
\\
\mathcal{I}_0&=(x_0^*,x_0^+), \quad x_0^*=\vartheta_{0,*}(0), \qquad 
x_0^+=\vartheta_{0,+}(0),
\end{aligned}
\end{equation}
with $y_n$ point of discontinuity of $\omega$.
Observe that $x_0^*>0$, and that
$\omega(0-)>\theta_l$, $f_l(\omega(0-))<f_r(\omega(0+))$, together imply
$\omega(0+)<\pi_{r,-}^l(\omega(0-))$. Hence, the states $\pi_{r,-}^l(\omega(0-)), \omega(0+)$
are connected by a shock with negative slope for the conservation law $u_t+f_r(u)_x$.
Thus, we will define the initial data $\overline u$ on $\mathcal{I}_0$ so to produce a compression
wave that generates a shock at $(0,T)$. Thus, we define
\vspace{-5pt}
\begin{equation}\label{initial control2}
\overline {u}(x):=\begin{cases}
\omega(y\pm)\ \ &\text{ if }\quad x\in \mathcal{I}_W,\ \
x=\vartheta_{y,\pm}(0),\ \  y\in \mathbb{R},\\
\noalign{\smallskip}
\omega(0-)\ \ &\text{ if }\quad x\in (x_0^-, 0),\ \
x_0^-=\vartheta_{0,-}(0),
\\
\noalign{\smallskip}
\pi_{r,-}^l(\omega(0-))\ \ &\text{ if }\quad x\in (0, x_0^*),\ \
x_0^*=\vartheta_{0,*}(0),
\\
\noalign{\smallskip}
(f'_l)^{-1}\big({(y_n-x)}/{T}\big)\ \ &\text{ if }\quad x\in (x_n^-,x_n^+)\subseteq \mathcal{I}_C,\ \
x_n^\pm=\vartheta_{y_n,\pm}(0),\ \ y_n<0,
\\
\noalign{\smallskip}
(f'_r)^{-1}\big({(y_n-x)}/{T}\big)\ \ &\text{ if }\quad x\in (x_n^-,x_n^+)\subseteq \mathcal{I}_C,\ \
x_n^\pm=\vartheta_{y_n,\pm}(0),\ \ y_n> 0,
\\
\noalign{\smallskip}
(f'_r)^{-1}\big({-x}/{T}\big)\ \ &\text{ if }\quad x\in (x_0^*,x_0^+)\subseteq \mathcal{I}_C,\ \
x_0^*=\vartheta_{0,*}(0),\ \ x_0^+=\vartheta_{0,+}(0).
\end{cases}
\end{equation}

\noindent
3. Then, setting for every $x\in\mathcal{I}_C$ :
\vspace{-5pt}
\begin{equation}
\eta_x(t):=
\begin{cases}
x+{\big((y_n-x)\,t\big)}/{T}
\quad &\text{if}\qquad x\in (x_n^-,x_n^+)\subseteq \mathcal{I}_C,\ \
x_n^\pm=\vartheta_{y_n,\pm}(0),\ \ y_n\neq 0,
\ \ t\in [0,T],\\
\noalign{\smallskip}
{-(x\,t)}/{T}
\quad &\text{if}\qquad x\in (x_0^*,x_0^+)\subseteq \mathcal{I}_C,\ \
x_0^*=\vartheta_{0,*}(0),\ \ x_0^+=\vartheta_{0,+}(0),
\end{cases}
\end{equation}

\vspace{-7pt}
\noindent
we define on $(\mathbb{R}\setminus\{0\})\times (0,T)$ the function:
\vspace{-5pt}
\begin{equation}\label{soluzione-costruita2}
u(x,t):=
 \begin{cases}
 \omega(y\pm)&\text{ if }\quad \exists~y \in\mathbb{R}:\ x={\vartheta}_{y,\pm}(t),\\
 \noalign{\smallskip}
  \omega(0-)&\text{ if }\quad \vartheta_{0,-}(t)<x<0,\\
  \noalign{\smallskip}
  \pi_{r,-}^l(\omega(0-))&\text{ if }\quad 0<x<\vartheta_{0,*}(t),\\
  \noalign{\smallskip}
  \overline u(y)&\text{ if }\quad \exists~y\in \mathcal{I}_C:\ x={\eta}_{y}(t)\,.
 \end{cases}
\end{equation}
\\

\noindent
4. Observe that, since $(\omega(0-),\omega(0+))\in \mathcal{T}_{3,-}$, it follows that 
the pair $u_l(t)=\omega(0-), u_r(t)=\pi_{r,-}^l(\omega(0-))$ satisfies the  interface entropy condition~\eqref{int-entr-cond3}.
Then, with the same arguments of the previous case, we conclude that $u$ is an $AB$-entropy solution  
of~\eqref{eq1},\eqref{flux}-\eqref{datum}, and that $\omega=S_T^{AB}\,\overline u$.
This proves that $\omega\in\mathcal{A}(T)$, and completes the proof of Theorem~\ref{Teo1}.
\qed

\section{Proof of Theorem~\ref{Teo2}}

The proof is devided in three steps.
\smallskip

\noindent
{\bf Step 1.}
Let $\mathcal{U}$ be as in~\eqref{insieme controlli} and let $\mathscr{C}\subset \mathscr{C}_f$ be a compact set of connections. Given $T>0$, $\lbrace\overline{u}_n\rbrace_n\subset\mathcal{U}$, 
and $(A,B)\in \mathscr{C}$, $\{(A_n,B_n)\}_n \subset \mathscr{C}$,
consider the sequences 
\vspace{-5pt}
\begin{equation}
\label{seq-comp}
\big\{\mathcal{S}^{AB}_T \overline u_n\big\}_n,
\qquad
\big\{\mathcal{S}^{A_n B_n}_T \overline u_n\big\}_n,
\qquad
\big\{\mathcal{S}^{AB}_{(\cdot)} \overline u_n\,\big|_{T}\big\}_n,
\qquad
\big\{\mathcal{S}^{A_n B_n}_{(\cdot)} \overline u_n\,\big|_{T}\big\}_n,
\end{equation}

\vspace{-8pt}
\noindent
where $u\big|_T$ denotes the restriction to $\mathbb{R}\times [0,T]$ of a map
defined on $\mathbb{R}\times [0,+\infty)$.
Since, $G$ in~\eqref{insieme controlli} is bounded,  $\mathscr{C}$ is compact 
and because of~\eqref{linf-bound-3} in Remark~\ref{propert3-AB-sol}, there holds
\vspace{-3pt}
 \begin{equation}\label{uniform bound121}
\big\| \mathcal{S}^{AB}_t \overline u_n\big\|_{\bf L^\infty(\mathbb{R})}\leq C,
\qquad
\big\| \mathcal{S}^{A_n B_n}_t \overline u_n\big\|_{\bf L^\infty(\mathbb{R})}\leq C
\qquad \forall~t\geq 0,\ \forall~n\,,
 \end{equation}

\vspace{-5pt}
\noindent for some constant $C>0$. Hence, the first two sequences in~\eqref{seq-comp}
are $weakly^\ast$ relatively compact in ${\bf L^\infty}(\mathbb{R})$,
the latter two are $weakly^\ast$ relatively compact in ${\bf L^\infty}(\mathbb{R}\times [0,T])$.
Thus, we can assume that
\vspace{-5pt}
\begin{equation}
\label{connect-conv}
\overline u_n \  \stackrel{\ast}{\rightharpoonup}\ \overline u\quad \textit{ in }\ \
 {\bf L^\infty}(\mathbb{R})
\,,\qquad\quad
(A_n, B_n) \ \to \ (\,\widetilde A, \widetilde B\,)\,,
\end{equation}

\vspace{-6pt}
\noindent
for some $\overline u\in {\bf L^\infty}(\mathbb{R})$, $(\,\widetilde A, \widetilde B\,)\in \mathscr{C}$, and that
 \begin{align}
 \mathcal{S}^{AB}_T \overline u_n &\  \stackrel{\ast}{\rightharpoonup}\  \omega^{AB},
 \qquad\quad
 \mathcal{S}^{A_n B_n}_T \overline u_n \ \stackrel{\ast}{\rightharpoonup}\ \omega^{\widetilde A \widetilde B}
 \quad \textit{ in }\ \
 {\bf L^\infty}(\mathbb{R}),
 \label{weak star convergence1}\\
 \mathcal{S}^{AB}_{(\cdot)} \overline u_n\,\big|_{T} &\ \stackrel{\ast}{\rightharpoonup}\ u^{AB},
 \qquad\quad
 \mathcal{S}^{A_n B_n}_{(\cdot)}  \overline u_n\,\big|_{T} \ \stackrel{\ast}{\rightharpoonup}\ \, u^{\widetilde A \widetilde B}
 \quad \textit{ in }\ \
 {\bf L^\infty}(\mathbb{R}\times [0,T]),
 \label{weak star convergence2}
 \end{align}
 for some functions $\omega^{AB}, \, \omega^{\widetilde A \widetilde B} \in {\bf L^{\infty}}(\mathbb{R})$ and 
 $u^{AB}, \, u^{\widetilde A \widetilde B}  \in {\bf L^{\infty}}(\mathbb{R}\times [0,T])$. 
 Notice that, since $\overline u_n(x)\in G(x)$ for almost every $x\in\mathbb{R}$, and because 
 $G$ is convex closed valued, by Mazur's lemma
  it follows from~\eqref{connect-conv} that $\overline u\in\mathcal{U}$. 
 We will show that there exist subsequences of~\eqref{seq-comp} that converge in the ${\bf L}^1_{loc}$
 topology to  $\omega^{AB}, \, \omega^{\widetilde A \widetilde B}$, and $u^{AB}, \, u^{\widetilde A \widetilde B}$, respectively, and that 
 \vspace{-5pt}
 \begin{equation}
 \label{compact-cond-1}
 \omega^{AB}=\mathcal{S}^{AB}_T \,\overline u, \qquad  \omega^{\widetilde A \widetilde B}=\mathcal{S}^{\widetilde A \widetilde B}_T \,\overline u\,,\qquad
 u^{AB}=\mathcal{S}^{AB}_{(\cdot)} \,\overline u\,\big|_{T}\,,\qquad
 u^{\widetilde A \widetilde B}=\mathcal{S}^{\widetilde A \widetilde B}_{(\cdot)} \,\overline u\,\big|_{T}\,,
 \end{equation}

\vspace{-4pt}
\noindent
 which proves the compactness of the sets $\mathcal{A}^{AB}\big(T,\mathcal{U}\big)$,
$\mathcal{A}\big(T,\mathcal{U},\mathscr{C}\big)$
and $\mathcal{A}^{AB}\big(\mathcal{U}\big), \, \mathcal{A}\big(\mathcal{U},\mathscr{C}\big)$.\\
 
 \noindent
{\bf Step 2.}
Notice that, by Remark~\ref{main-thm-interpret1}, for any $0<a<b$, 
there exists $C_{a,b}, L_{a,b}>0$ such that,
setting $I_{a,b}\doteq  [-b,-a]\cup [a,b]$, one has
\begin{equation*}
\text{Tot.\!Var.}\big\{
\mathcal{S}^{AB}_t \overline u_n: I_{a,b}\big\}\leq C_{a,b},
\qquad
\text{Tot.\!Var.}\big\{
\mathcal{S}^{A_n B_n}_t \overline u_n: I_{a,b}\big\}\leq C_{a,b},
 \qquad \forall~t\in [a,T],\ \forall~n,
\end{equation*}
 \vspace{-15pt}
\begin{equation*}
\big\|\mathcal{S}^{AB}_t \overline u_n-\mathcal{S}^{AB}_s \overline u_n
\big\|_{\bf L^1(I_{a,b})}\leq L_{a,b}\cdot |t-s|,\qquad\quad
\big\|\mathcal{S}^{A_n B_n}_t \overline u_n-\mathcal{S}^{A_n B_n}_s \overline u_n
\big\|_{\bf L^1(I_{a,b})}\leq L_{a,b}\cdot |t-s|\qquad \forall~t,s\in [a,T],\ \forall~n.
\end{equation*}
By Helly's theorem there exists subsequences $\big\{\mathcal{S}^{AB}_t \overline u_{n_j}\big\}_j$\,,
$\big\{\mathcal{S}^{A_{n_j} B_{n_j}}_t \overline u_{n_j}\big\}_j$\,, 
which converges to some functions $w (\cdot,t)$ and $\widetilde w (\cdot,t)$, respectively,
 in  ${\bf L^1}(I_{a,b})$ for all $t\in[0,T]$. 
 Because of~\eqref{weak star convergence2}, the functions $w, \widetilde w$ must coincide with
 the restriction to $I_{a,b}\times [0,T]$
 of $u^{AB}$
 and $u^{\widetilde A, \widetilde B}$, respectively, and there holds
  \vspace{-5pt}
 \begin{equation*}
 \mathcal{S}^{AB}_t \overline u_{n_j}\ \to \ u^{AB}(\cdot,t)\,,
 \qquad\quad
 \mathcal{S}^{A_{n_j} B_{n_j}}_t \overline u_{n_j}\ \to \ u^{\widetilde A\widetilde B}(\cdot,t)\quad
 \textit{in}\quad {\bf L^1}(I_{a,b})\,, 
 \qquad  \forall~t\in [a,T]\,.
 \end{equation*}
%
 Then, repeating the same arguments for $I_{a_j,b_j}\doteq  [-b_j,-a_j]\cup [a_j,b_j]$, with $a_j\downarrow  0, \,
 b_j\to+\infty$, and observing that by~\eqref{uniform bound121} one has $\|u^{AB}(\cdot,t)\|_{\bf L^\infty(\mathbb{R})}
 \leq C$, $\|u^{\widetilde A \widetilde B}(\cdot,t)\|_{\bf L^\infty(\mathbb{R})}
 \leq C$, for all $t\in [0,T]$, we deduce that
 we can select diagonal subsequences (still denoted with index $j$) such that
 \begin{equation}
 \label{elleuno-conv-compact-21}
 \mathcal{S}^{AB}_t \overline u_{n_j}\ \to \ u^{AB}(\cdot,t)\,,
 \qquad\quad
 \mathcal{S}^{A_{n_j} B_{n_j}}_t \overline u_{n_j}\ \to \ u^{\widetilde A\widetilde B}(\cdot,t)\quad
 \textit{in}\quad {\bf L}^1_{loc}(\mathbb{R})\,, 
 \qquad  \forall~t\in (0,T]\,.
 \end{equation}
 %
 In particular, because of~\eqref{weak star convergence1}, \eqref{elleuno-conv-compact-21},
 we have $u^{AB}(\cdot,T)=\omega^{AB}$, $u^{\widetilde A \widetilde B}(\cdot,T)=\omega^{\widetilde A \widetilde B}$.
 Therefore, in order to establish~\eqref{compact-cond-1}, it remains to show only that
  \vspace{-2pt}
 \begin{equation}
 \label{compact-cond-2}
 u^{AB}=\mathcal{S}^{AB}_{(\cdot)} \,\overline u\,\big|_{T}\,,\qquad
 u^{\widetilde A \widetilde B}=\mathcal{S}^{\widetilde A \widetilde B}_{(\cdot)} \,\overline u\,\big|_{T}\,.
 \end{equation}
 We will provide only a proof of the second equality in~\eqref{compact-cond-2}, the proof of the first one being
 entirely similar.\\
  
 \noindent
{\bf Step 3.}
First observe that, by the regularity of $f_l,f_r$, the convergence~\eqref{elleuno-conv-compact-21} 
implies that 
 \vspace{-5pt}
\begin{equation}
 \label{elleuno-conv-compact-26}
 \begin{aligned}
 f_l\big(\mathcal{S}^{A_{n_j} B_{n_j}}_t \overline u_{n_j}\big)\ \to \ 
 f_l\big(u^{\widetilde A\widetilde B}(\cdot,t)\big)
 \quad &\textit{in}\quad {\bf L}^1_{loc}((-\infty,0])
 \\ 
 f_r\big( \mathcal{S}^{A_{n_j} B_{n_j}}_t \overline u_{n_j}\big)\ \to \ 
 f_r\big(u^{\widetilde A\widetilde B}(\cdot,t)\big)\quad
 &\textit{in}\quad {\bf L}^1_{loc}([0,+\infty)\,,
 \end{aligned} 
 \qquad  \forall~t\in (0,T]\,.
 \end{equation}
 
\vspace{-4pt}
\noindent
Therefore, since $u_{n_j}(\cdot,t)\doteq \mathcal{S}^{A_{n_j} B_{n_j}}_t \overline u_{n_j}$, $t\geq 0$,
are in particular weak solutions of the Cauchy problem for~\eqref{eq1},\eqref{flux}
with initial data $\overline u_{n_j}$, 
relying on~\eqref{elleuno-conv-compact-21}, \eqref{elleuno-conv-compact-26}, 
and on~\eqref{connect-conv}, we find
 \vspace{-5pt}
\begin{equation}
\label{weak-sol-cauchy}
\begin{aligned}
&\int_{-\infty}^{\infty}\int_0^{\infty}
 \Big\{{u^{\widetilde A \widetilde B}\phi_t }+f\big(x,u^{\widetilde A \widetilde B}\big)\phi_x
 \Big\}dxdt
 +\int_{-\infty}^{\infty}{\overline u(x)\phi(x,0)}dx=
 \\
 &\quad =\lim_{j\to\infty} \int_{-\infty}^{\infty}\int_0^{\infty}
 \big\{{u_{n_j} \phi_t }+f\big(x,u_{n_j}\big)\phi_x
 \big\}dxdt
 +\int_{-\infty}^{\infty}{\overline u_{n_j}(x)\phi(x,0)}dx~=~0\,,
 \end{aligned}
\end{equation}

\vspace{-4pt}
\noindent
for any test function $\phi\in\mathcal{C}^1_c$ with compact support 
contained in $\mathbb{R}\times (0,+\infty)$, 
which shows that
$u^{\widetilde A \widetilde B}$ is a weak solution of the Cauchy problem~\eqref{eq1},\eqref{flux}-\eqref{datum}. 
Next, setting $I_l\doteq (-\infty, 0)$, $I_r\doteq (0,+\infty)$, with the same arguments we derive
\vspace{-5pt}
 \begin{align*}
\int_{I_{l,r}}\int_0^{+\infty}
&\Big\{\big|u^{\widetilde A \widetilde B}-k\big|\phi_t +\big(f_{l,r}\big(u^{\widetilde A \widetilde B}\big)-f_{l,r}(k)\big)\sgn{\big(u^{\widetilde A \widetilde B}-k\big)}\phi_x\Big\}dxdt=\\ 
\noalign{\vspace{-5pt}}
&\lim_{j\rightarrow \infty}
\int_{I_{l,r}}\int_0^{+\infty} \Big\{\big|u_{n_j}-k\big|\phi_t +\big(f_{l,r}\big(u_{n_j}\big)-f_{l,r}(k))\sgn \big(u_{n_j}-k\big)\phi_x\Big\} dxdt~\geq~0\,,
\end{align*}

\vspace{-4pt}
\noindent
for any non negative function $\phi\in\mathcal{C}^1$ with compact support in $I_{l,r}\times (0,T]$ and for any $k\in\mathbb{R}$. Therefore, since $u^{\widetilde A \widetilde B}$ is a weak solution
of the Cauchy problem~\eqref{eq1},\eqref{flux}-\eqref{datum}, that satisfies the 
Kru\v{z}hkov entropy inequalities on $(\mathbb{R}\setminus\{0\})\times (0,T]$,
invoking a result in~\cite{MR1771520} (see also~\cite[Coroll. 6.8.4]{MR3468916})
we deduce that the map $t\to u^{\widetilde A \widetilde B}(\cdot,t)$ is continuous from $[0,T]$
in $\mathbf{L}^1_{loc}(\mathbb{R})$, and that the initial condition~\eqref{datum} is satisfied.
This shows that  $u^{\widetilde A \widetilde B}$ satisfies conditions (i)-(ii) of Definition~\ref{def-ab-entr-sol}.

Finally, observing that by definition~\eqref{AB-adapt-entr} 
and because of~\eqref{connect-conv}, there holds
$k_j\doteq k_{A_{n_j}\,B_{n_j}} \to k_{\widetilde A\, \widetilde B}$ in ${\bf L}^1_{loc}(\mathbb{R})$,
we deduce that $u^{\widetilde A \widetilde B}$ satisfies also the 
Kru\v{z}hkov-type entropy inequality associated to the $\widetilde A \widetilde B$-connection.
Namely,   for any non negative function $\phi\in  \mathcal{C}^1$ with compact support in $\mathbb{R}\times (0,T]$, 
we get
\vspace{-5pt}
 \begin{align*}
 \int_{-\infty}^{+\infty}\int_{0}^{\infty}&
 \big\{\big|u^{\widetilde A \widetilde B}-k_{\widetilde A \widetilde B}(x)\big|\phi_t+\big(f\big(x,u^{\widetilde A \widetilde B}\big)-f(x,k_{\widetilde A \widetilde B}(x))\big)\sgn\big(u^{\widetilde A \widetilde B}-k_{\widetilde A \widetilde B}(x)\big)\phi_x
 \big\}dxdt=\\
 \noalign{\vspace{-5pt}}
&\lim_{j\rightarrow \infty}\int\int \big\{\big|u_{n_j}-k_j(x)\big|\phi_t +\big(f\big(u_{n_j}\big)-f\big(k_j(x)\big)\sgn\big(u_{n_j}-k_j(x)\big)\phi_x\big\}dxdt~\geq~0,
 \end{align*}
 
 \vspace{-5pt}
\noindent
which shows that $u^{\widetilde A \widetilde B}$ is an $\widetilde A \widetilde B$-entropy solution
of the Cauchy problem~\eqref{eq1},\eqref{flux}-\eqref{datum} 
on $\mathbb{R}\times [0,T]$, according with definition~\eqref{AB-adapt-entr} .
Thus, by uniqueness of $\widetilde A \widetilde B$-entropy solutions of the Cauchy problem (see Theorem~\ref{semigroup-AB}), we deduce that  
$u^{\widetilde A \widetilde B}=\mathcal{S}^{\widetilde A \widetilde B}_{(\cdot)} \,\overline u\,\big|_{T}$,
completing the proof of Theorem~\ref{Teo2}.
\qed

\section{Some applications in
LWR  traffic flow models}

Starting from the seminal papers by Lighthill, Whitham~\cite{MR72606} and Richards~\cite{MR75522}, 
the evolution of unidirectional traffic flow along an highway can be described 
at a macroscopic level
with a
partial differential equation (LWR model) where the dynamical variable is the traffic density $\rho(x,t)$
(the number of vehicles per unit length). The LWR model expresses the mass conservation,
i.e. the conservation of the total number of vehicles, and postulates that the average traffic speed $v(x,t)$ is
a function of the traffic density alone.
Thus, the mean traffic flow (the number of cars crossing the point $x$ per unit time)
is given by $f(x,t)=\rho(x,t)\,v(\rho(x,t))$, and we are lead to the hyperbolic conservation law
\vspace{-5pt}
\begin{equation}\label{equation for cars} 
\rho_t+ (\rho\, v(\rho))_x=0\,.
\end{equation}
Here 
$\rho(x,t)$ takes values in the interval $[0, \rho_{max}]$, where 
$\rho_{max}$ represents the situation
in which the vehicles are bumper to bumper and thus depends only on the average length of
the vehicles. The velocity $v(\rho)$ has a maximum value $v_{max}$
(representing the limit speed) attained at $\rho=0$, and 
it  is strictly decreasing
since in presence of larger number of cars each driver goes slower. 
Hence, 
the corresponding flux $f(\rho)=\rho \, v(\rho)$ (the so-called {\it fundamental diagram}) 
is a (uniformly) strictly concave map (see Figure~\ref{figure  14}), satisfying the  assumptions 
\vspace{-13pt}
 \begin{enumerate}
\item[{\bf H1)$^\prime$}] $f_l, f_r: \mathbb{R}
\rightarrow \mathbb{R}$ 
are twice continuously differentiable, (uniformly) strictly concave maps
\vspace{-3pt}
$$\max\big\{f''_{l}(u)\,\, f''_r(u)\big\}\leq -c<0\qquad \forall~u\in \mathbb{R},$$
\end{enumerate}

\vspace{-5pt}
\noindent
and {\bf H2)-H3)} in Section~\ref{sec-prelim} (with $\rho_{max}$ in place of $1$).
We refer to~\cite{MR3553143,MR3110698,MR3076426} 
for general references on macroscopic models of traffic flow.
\pagebreak

\vspace{-3pt}
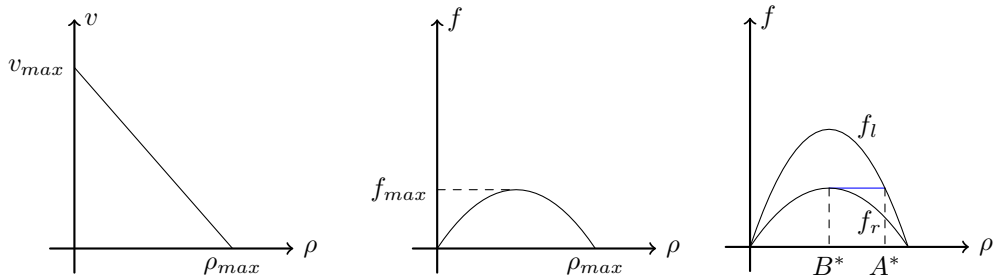
\begin{figure}[!htbp] \label{figure LWR}
\begin{center}
\begin{tikzpicture}
[scale=1.6,cap=round]
 \def\costhirty{0.8660256}

\tikzstyle{axes}=[]
  \tikzstyle{important line}=[very thick]
  \tikzstyle{information text}=[rounded corners,fill=red!10,inner sep=1ex]

\draw[thick,->] (-0.2,0) -- (1.8,0) node[right] {$\rho$};
\draw[thick,->] (0,-0.2) -- (0,1.9) node[right] {$v$};

\draw (0,1.5) node[anchor=east] {$v_{max}$};
\draw (1.3,0) node[anchor=north] {$\rho_{max}$};

\draw (0,1.5) -- (1.3, 0);
\end{tikzpicture}
\quad
\begin{tikzpicture}
[scale=1.6,cap=round]
 \def\costhirty{0.8660256}

\tikzstyle{axes}=[]
  \tikzstyle{important line}=[very thick]
  \tikzstyle{information text}=[rounded corners,fill=red!10,inner sep=1ex]

\draw[thick,->] (-0.2,0) -- (1.8,0) node[right] {$\rho$};
\draw[thick,->] (0,-0.2) -- (0,1.9) node[right] {$f$};

\draw (0, 1125/2308) node[anchor=east] {$f_{max}$};
\draw (1.3,0) node[anchor=north] {$\rho_{max}$};
\draw[dashed] (0, 1125/2308) -- (375/577, 1125/2308);
\draw [black,domain=0:1.3] plot (\x, {- 1.154*pow(\x,2)+1.5*\x });

\end{tikzpicture}
\quad
\begin{tikzpicture}
[scale=1.6,cap=round]
 \def\costhirty{0.8660256}

\tikzstyle{axes}=[]
  \tikzstyle{important line}=[very thick]
  \tikzstyle{information text}=[rounded corners,fill=red!10,inner sep=1ex]

\draw[thick,->] (-0.2,0) -- (1.8,0) node[right] {$\rho$};
\draw[thick,->] (0,-0.2) -- (0,1.9) node[right] {$f$};

\draw (0.8,1) node[anchor=west] {$f_{l}$};
\draw (0.8,0.2) node[anchor=west] {$f_{r}$};
\draw[blue] (0.65, 1125/2308) -- (1.109, 1125/2308);
\draw[dashed]  (1.109, 1125/2308) --  (1.109, 0);
\draw[dashed]  (0.65, 1125/2308) --  (0.65, 0);

\draw (0.65, 0) node[anchor=north] {$B^\ast$};
\draw (1.109, 0) node[anchor=north] {$A^\ast$};

\draw [black,domain=0:1.3] plot (\x, {- 1.154*pow(\x,2)+1.5*\x });
\draw [black,domain=0:1.3] plot (\x, {- 1.154*2*pow(\x,2)+1.5*2*\x });
\end{tikzpicture}

\end{center}
\vspace{-10pt}
\caption{Velocity and flux in the LWR model, and a discontinuous flux with critical connection
\label{figure  14}} 
\end{figure}
%
%
The occurrence of  special events (like heavy rain) 
that alter the road condition, or the presence along the
road of obstacles
(such as speed bumps, construction sites)
that hinder the traffic flow,
may force the vehicles to slow down or speed up in different
sections of the highway.
These inhomogeneities of the road are described
by considering different speed-density relationships (and therefore different  fundamental diagrams) 
on different portions of  the highway. Assuming for simplicity that the change in the flow-density relation 
in two sections of the road of infinite length occurs at
$x=0$, we are led to a conservation law with discontinuous flux
$f(x,\rho)$ as in~\eqref{flux},
 where the right and left fluxes are 
 of the form $f_{l,r}(\rho)=\rho\,v_{l,r}(\rho)$. 
This model was considered in~\cite{MR898784} where it was employed 
an admissibility criterion for the one-sided limits of the solution at $x=0$
according with the flux maximization principle.
Such a criterion
is equivalent to an interface entropy condition
as in~\eqref{int-entr-cond3}
relative to a critical connection $(A^*,B^*)$
 passing trough the minimum of the two points of maximum of $f_l, f_r$
 (see Figure~\ref{figure  14}).
 Here, since we are considering a two-flux concave flux, we replace in
 the $AB$-interface entropy condition~\eqref{int-entr-cond3}
 the $\leq$ signs with the $\geq$ signs and viceversa. This implies that the 
 flux of an $AB$-entropy solution along the discontinuity $\{x=0\}$ must be smaller or equal to the 
 value of the flux on the connection.
 We let ${{\mathcal{S}}}^{*}$ denote the solution operator 
for~\eqref{eq1}, \eqref{flux} with  fluxes $f_{l,r}(\rho)=\rho\,v_{l,r}(\rho)$,
$\rho\in [0, \rho_{max}]$,
and connection $(A^*, B^*)$.
Since our analysis will be focused on a finite section of the road,
we shall assume that all the initial data have support in a bounded set $K\subset \mathbb{R}$.
One can derive similar characterization of the attainable set provided by Theorem~\ref{Teo1} in the two-flux concave case.
Thus, the results stated in  Theorem~\ref{Teo2}
and Corollary~\ref{existence-opt-sol}
continue to hold as well in the concave-concave case.

In this setting we shall consider two type of optimization problems. In the first one 
we treat as control parameters only the initial data.
Instead, in the latter we regard as control parameters also the connection states whose
flux value provides an upper  limitation on the flux of the solution
at $x=0$. 
Such a control can be viewed as a {\it local point constraint control}  acting at $x=0$ (cfr.~\cite{MR2817547}).
Similar problems in the context of a junction are treated in~\cite{traffic-index}.
\smallskip

 {\bf Output least square optimization with traffic density target.}
 In order to validate the LWR models employed by
transport engineers, it is fundamental to compare the experimental data with 
the solutions that better approximate a given observable function. 
A classical cost functional adopted to this pourpose is the ${\bf L^2}$-distance from
an observation output (see for instance~\cite{MR1680830}).
Thus, we are led to consider the 
optimization problem 
\vspace{-5pt}
 \begin{equation}
 \label{output-opt-traffic}
\min_{\overline \rho\in\mathcal{U}}\int_{\mathbb{R}}|S^*_T \overline\rho (x)-l(x)|^2dx,
\end{equation} 

\vspace{-8pt}
\noindent
where $\mathcal{U}$ is the admissible control set 
\vspace{-3pt}
 \begin{equation}
 \label{indata-oil-reservoir}
 \mathcal{U}\doteq \Big\{
 \overline u\in \mathbf{L}^\infty(\mathbb{R});\
 \overline u(x)\in G(x)\ \ \text{for a.e.}~~x\in\mathbb{R}
 \Big\},
 \end{equation}
 with
 \vspace{-5pt}
 \begin{equation}
 \label{G-def}
 G(x)=
 \begin{cases}
 [0,\rho_{max}]\quad &\text{if}\quad x\in K\,,
 \\
 \{0\} &\text{otherwise},
 \end{cases}
 \end{equation}
and $ l\in L^{2}(\mathbb{R})$ is a given target function.
Notice that, by Remark~\ref{propert3-AB-sol}, there will be some bounded set $K'\subset \mathbb{R}$
such that
\vspace{-5pt}
\begin{equation}
\label{Omega-def}
{\mathcal{S}}^{*}_t \,\overline s\in
\Omega\doteq
\Big\{
 \omega\in \mathbf{L}^\infty(\mathbb{R}; [0,c]);\
 \text{supp}(\omega) \subset K'
 \Big\}
 \qquad\forall~\overline s \in \mathcal{U}, \ t\geq 0\,.
\end{equation}
Therefore, since the map $\omega \mapsto \int_{\mathbb{R}} |\omega(x)-l(x)|^2dx$
is clearly continuous on $\Omega$ with respect to the $\mathbf{L}^1(\mathbb{R})$
topology, we deduce 
the existence of a solution 
to problem~\eqref{output-opt-traffic}
from the natural extension of Corollary~\ref{existence-opt-sol} to the two-flux concave case.

Alternatively, in order to address road safety issues in planning design, 
 it is important to analyse the initial density distributions and the (upper) flow limitations 
at the flux discontinuity interface which lead to the closest configuration to a desired density distribution.
For example, one may consider two stretches of road of different capacities
connected at a junction located in front of a school, where one may 
regulate the maximum rate at which the vehicles pass through the junction.
In this case, it would be interesting to analyze the solutions of the optimization problem
\vspace{-5pt}
 \begin{equation}
 \label{output-opt-traffic-2}
\min_{\overline \rho\in\mathcal{U},\,(A,B)\in\mathscr{C}}\int_{\mathbb{R}}|S^{AB}_T \overline\rho (x)-l(x)|^2dx,
\end{equation} 
where $T$ is the exit time from school,  $l\in L^{2}(\mathbb{R})$ represents a ``safe'' traffic distribution,
$\mathcal{U}$ is the set of admissible initial data as above, and $\mathscr{C}$ is a compact set of connections.
Again, relying on 
the analogous result of Corollary~\ref{existence-opt-sol} for the two-flux concave case, 
we deduce the existence of a solution 
to~\eqref{output-opt-traffic-2}.
\\

 {\bf Fuel consumption optimization.}
 Traffic simulation is  a fundamental instrument to predict the impact of road design 
 and to examine the performance of traffic facilities under changing surface conditions. 
In this context, a major challenge for transport planners is to design solutions for mitigating pollution,
which has huge economic impact, beside affecting people's quality of life.
Various definitions to quantify the overall fuel consumption have been introduced
in the literature (see~\cite{tk}).
We employ here the definition proposed in~\cite{seibold}
where the fuel consumption rate of a single vehicle 
is expressed by a 
polynomial function $P$ depending only on the average traffic speed $v(\rho)$.
The overall fuel consumption rate is then obtained multiplying $P$ by the density $\rho$.
Thus, if we consider two stretches of road of different capacities
connected at a junction where we may regulate the maximum flow rate of traffic, and we are interested in analyzing the initial density distribution that produces the minimum fuel consumption in a given interval of time $[0,T]$,
we are led to the optimisation problem
\vspace{-5pt}
\begin{equation}\label{funzionale 2}
\min_{\overline \rho\in\mathcal{U},\,(A,B)\in\mathscr{C}}\int_0^T\int_{\mathbb{R}}  S_t^{AB} \overline\rho(x) P\big(v( S_t^{AB} \overline\rho(x))\big)dx dt\,,
\end{equation}
with $\mathcal{U}$ and $\mathscr{C}$ as above.
Observe that, by Remark~\ref{propert3-AB-sol}, there will be some bounded set $K'\subset \mathbb{R}$
such that
\begin{equation}
\label{Omega-def2}
{\mathcal{S}}^{AB}_{(\cdot)} \,\overline \rho\in
\Omega\doteq
\Big\{
 \omega\in \mathbf{L}^\infty(\mathbb{R}\times[0,T]; [0,\rho_{max}]);\
 \text{supp}(\omega) \subset K'
 \Big\}
 \qquad\forall~\overline \rho \in \mathcal{U}, \ (A,B)\in \mathscr{C}\,.
\end{equation}
Hence, since the map $\omega \mapsto \int_{\mathbb{T}\times [0,T]}
\omega(x,t) P\big(v(\omega(x,t))\big)dx dt$ is  
 continuous on $\Omega$ with respect to the $\mathbf{L}^1(\mathbb{R}\times [0,T])$
topology, we deduce the existence of a solution to~\eqref{funzionale 2}
from the analogous result of Corollary 2.3 for the two-flux concave case.

\bibliographystyle{siam}
\bibliography{final_reference}

\end{document}